\documentclass[10pt]{amsart}
\usepackage{latexsym,amsfonts,amssymb,amsthm,amsmath,mathrsfs,color,amscd,graphicx,fullpage,parskip,hyperref,bm,bbm,dsfont}
\usepackage[T1]{fontenc}
\usepackage{pgfplots,tikz}

\usepackage{graphicx}
\usepackage{caption}
\usepackage{float}

\usepackage{comment}
\includecomment{uncomment}
\usepackage{commath,mathtools,setspace}
\usepackage{paralist}
\usepackage{extarrows}

\excludecomment{comment} % for comments we want to exclude

\newcommand{\s}[1]{{\mathcal #1}}

\newcommand{\bb}[1]{{\mathbb #1}}
\newcommand{\vd}[2]{\dfrac{\delta #1}{\delta #2}}

\newcommand{\ip}[2]{\left\langle #1,#2 \right\rangle}

\newcommand{\firststep}{\setcounter{step}{1}\textbf{Step \arabic{step}.} }
\newcommand{\nextstep}{\stepcounter{step}\textbf{Step \arabic{step}.} }

\DeclareMathOperator{\argmin}{argmin}

\newcommand{\vphi}{{\varphi}}

\DeclareMathOperator{\spt}{spt}
\DeclareMathOperator{\range}{range}

\usepackage{mathtools}
\mathtoolsset{showonlyrefs}

\newcommand{\sP}{{\mathscr P}}

\newcommand{\bH}{{\mathbb H}}

\newcommand{\R}{{\mathbb R}}

\newcommand{\ds}{\displaystyle}

\newtheorem{theorem}{Theorem} %[section]
\newtheorem{corollary}[theorem]{Corollary}

\newtheorem{lemma}[theorem]{Lemma}

\newtheorem{definition}[theorem]{Definition}

\newtheorem{remark}[theorem]{Remark}
\newtheorem{assumption}[theorem]{Assumption}

\numberwithin{equation}{section}
\numberwithin{theorem}{section}

\newcounter{step}
\begin{document}
	
	\title[Mean Field Games and Conservation Laws]{On Some Mean Field Games and Master Equations through the lens of conservation laws}
	%{Some possible connections between mean field games and entropy solutions to conservation laws}
	
	\author{P.~Jameson Graber}
	%\thanks{National Science Foundation under NSF Grants DMS-1612880 and DMS-1905449.}
	\address{J. Graber: Baylor University, Department of Mathematics;\\
		Sid Richardson Building\\
		1410 S.~4th Street\\
		Waco, TX 76706\\
		Tel.: +1-254-710- \\
		Fax: +1-254-710-3569 
	}
	\email{Jameson\_Graber@baylor.edu}
	
	\author{Alp\'ar R. M\'esz\'aros}
	\address{A.R. M\'esz\'aros: Department of Mathematical Sciences, Durham University, Durham DH1 3LE, United Kingdom}
	\email{alpar.r.meszaros@durham.ac.uk} 
	
	%\subjclass[2010]{35Q91, 35F61, 49J20}
	\date{\today}   
	
	%\begin{comment}
	\begin{abstract}
		In this manuscript we derive a new nonlinear transport equation written on the space of probability measures that allows to study {a class of} deterministic mean field games and master equations, {where the interaction of the agents happens only at the terminal time}. The point of view via this transport equation has two important consequences. First, this equation reveals a new monotonicity condition that is sufficient both for the uniqueness of MFG Nash equilibria and for the global in time well-posedness of master equations. Interestingly, this condition is in general in dichotomy with both the Lasry--Lions and displacement monotonicity conditions, studied so far in the literature. 
		Second, in the absence of monotonicity, the conservative form of the transport equation can be used to define weak entropy solutions to the master equation. We construct several concrete examples to demonstrate that MFG Nash equilibria, whether or not they actually exist, may not be selected by the entropy solutions of the master equation.
	\end{abstract}
	
	%\end{comment}

	\maketitle
	
	%\tableofcontents
	
	\section{Introduction}
	
The history of game theoretical models with an infinite number or continuum of agents in the economical literature dates back to the 1960's, and these were first investigated independently by Aumann and Shapley (see \cite{aumann64} and the references therein). Roughly half a century later, two groups, Lasry--Lions (cf. \cite{lasry06,lasry06a,lasry07}) and Huang--Malham\'e--Caines (cf. \cite{huang2006large}) were interested in characterizing limits of Nash equilibria of stochastic differential games, when the number of agents tends to infinity. These studies gave birth to the theory of \emph{mean field games} (MFG), which in the past 15 years or so has found a great number of applications in many different fields and initiated profound mathematical research. 
	
There are two main approaches when it comes to rigorously study MFGs: one uses probabilistic tools and often relies on studying systems of FBSDEs, while the other is an analytic one which is based on the study of PDE (or SPDE) systems. At this level, the MFG system consists of a Hamilton--Jacobi--Bellman (HJB) equation, written for the value function of a typical agent, and a Fokker--Planck equation describing the evolution of the density of the agent population. For an excellent introduction to the two approaches we refer to \cite{carmona2018probabilistic} and \cite{cardaliaguet2020intro}, respectively.
	
A fundamental object in this theory is the \emph{master equation}, introduced by Lions in \cite{lions07}. This is a nonlinear and nonlocal PDE set on the space of probability measures that encodes all the information about the game. This equation has a great significance in rigorously showing that Nash equilibria of $N$-player stochastic differential games converge as $N\to+\infty$ to their mean field limit (cf. \cite{cardaliaguet2019master}). In particular, smooth solutions to the master equation will also provide quantified rates of convergence. The question of well-posedness of the master equation initiated an important program in the field. Local in time classical solutions to the master equation are known to exist and be unique under almost no structural assumptions on the data, even in the presence of common noise or in completely deterministic settings (\cite{gangbo2015existence,mayorga2019short,carmona2018probabilisticII,ambrose2021well,cardaliaguet2022splitting}). The question of global in time well-posedness of the master equation is more subtle. From the point of view of classical solutions this is known under additional {\it monotonicity} assumptions on the data. Roughly speaking, such assumptions prevent the crossing of generalized characteristics, and as a result the well-posedness of the master equation can be obtained. As the MFG system plays the role of generalized characteristics for the master equation, the existence of classical solutions to the latter one is intimately linked to the question of uniqueness of solutions to the former one.

\noindent {\bf Master equations in the known monotone regimes.} In the current literature we can distinguish essentially two directions: the so-called {\it Lasry--Lions} (LL) monotonicity condition, which corresponds to `flat interpolation' of probability measures (the geometry associated to the $W_1$ Wasserstein metric), and the so-called {\it displacement monotonicity} condition that stems from the notion of displacement convexity arising in the theory of optimal transport; this latter one corresponds to interpolations of probability measures along $W_2$-geodesics. Historically, the LL monotonicity condition appeared first in the context of MFG, this being also the first sufficient condition for uniqueness of (regular enough) solutions to the MFG system (cf. \cite{lasry07}). In the context of the master equation, the first fundamental global in time well-posedness results (\cite{chassagneux, cardaliaguet2019master,carmona2018probabilisticII}) relied on the LL monotonicity condition in the case of so-called separable Hamiltonians (the momentum and measure variables being additively separated) and non-degenerate idiosyncratic noise. This monotonicity condition (in the case of separable Hamiltonians) was used later in \cite{bertucci2021monotone, mou2019wellposedness} to define various notions of weak solutions to the master equation. In these references the underlying MFG system still had unique solutions. Displacement monotonicity conditions in context of MFG have first been used in \cite{ahuja2016wellposedness} (see also \cite{ahuja2019forward, chassagneux, carmona2018probabilisticII}). However, it became evident only in the works \cite{gangbo2020global,gangbo2021mean, bensoussan2019stochastic, bensoussan2020control, meszaros2021mean} that displacement monotonicity could serve as an alternative sufficient condition for the global in time well-posedness of master equations not only in the absence of non-degenerate idiosyncratic noise (cf. \cite{gangbo2020global}), but in addition in the case of a general class of non-separable Hamiltonians (cf. \cite{gangbo2021mean, meszaros2021mean}). In \cite{gangbo2020global} it has also been pointed out that displacement monotonicity is in general in dichotomy with the LL monotonicity condition.
	
At this point, let us emphasize that both the LL and displacement monotonicity conditions are only sufficient conditions in the corresponding cases both for the uniqueness of (regular enough) solutions to MFG systems and for the global in time existence of classical solutions to the corresponding master equations. Any alternative condition that could ensure uniqueness and stability of solutions to the MFG system could potentially imply the corresponding global in time well-posedness theory for the underlying master equations. In this manuscript, one of our main objectives is to propose a {\it new monotonicity condition}, which is yet another sufficient condition both for the uniqueness of solutions to MFG systems and the global in time well-posedness of the corresponding master equations. 
	We do { this for a class of deterministic MFGs, where agent interactions happen only at the terminal time. Under such structural restrictions, we relate the study of MFG Nash equilibria to finite dimensional conservation laws. Hence, we find the new monotonicity condition} by observing that, in the theory of nonlinear conservation laws, the maximal time interval for the existence of classical solutions is found by analyzing the composition of wave speeds with the initial condition, cf.~\cite[Theorem 6.1.1]{dafermos16}.
	The analog in mean field games is to analyze the optimal flow for individual players composed with a particular quantity that is ``transported'' by the game; see the summary of our contributions below for a formal explanation.
	From this point of view, it is possible to derive a new monotonicity condition for uniqueness of the equilibrium, which is
 in dichotomy with both the LL and displacement monotonicity conditions: in general it implies neither of these two, nor is it implied by one or the other. For a detailed account on the properties of this new class of monotonicity condition we refer to our accompanying paper \cite{graber2022unique}. So, as a consequence of this new monotonicity condition, we show the global in time well-posedness of the corresponding master equation (from the point of view of classical solutions, {in the case of the model examples involving agent interactions only at the terminal time}) and this is our first contribution in this paper. In the deterministic setting, which we consider here, global in time well-posedness of the master equation has been known in the literature before essentially in the displacement monotone setting only.

{\bf The literature related to multiple Nash equilibria in MFG and selection mechanisms.} A still poorly understood direction in the theory of MFG and master equations is when the underlying game has multiple Nash equilibria. In such a scenario, it is evident that one needs to abandon the notion of classical solutions for the corresponding master equations, as non-uniqueness of equilibria implies the crossing of generalized characteristics. In such cases, in general it is not evident what kind of selection criteria to apply to pick a particular equilibrium from the set of all equilibria, which equilibrium choice is better motivated from the point of view of economical applications, etc. Evidently, the well-posedness question of master equations in a suitable weak sense poses even greater mathematical challenges. A related question is whether it would be possible to restore uniqueness of Nash equilibria by adding a suitably chosen random forcing. In the past couple of years there has been a significant effort to investigate these questions. Interesting non-uniqueness examples (in the case of nonlocal coupling functions that violate the LL monotonicity contion) were constructed in \cite{bardi2019on, briani2018stable, cirant2019on}. In particular in \cite{briani2018stable} the authors have found `stable solutions' which could be selected via a learning procedure. In \cite{cecchin2019concergence} the authors consider a MFG that has a state space with two elements, and which has three closed loop Nash equilibria. The equilibria of the associated $N$-player game select one particular MFG equilibrium in the limit as $N\to+\infty$ and the value functions converge as $N\to+\infty$ to the entropy solution of a one dimensional conservation law. Selection mechanisms were designed in the works \cite{tchuendom2018uniqueness, delarue2020selection}, where the authors study some classes of linear quadratic MFG problems, with multiple equilibria. They show that by adding an appropriate common noise, the uniqueness can be restored in the corresponding MFGs. Moreover, the vanishing noise limit will select certain equilibria. Selection in the latter work is provided also via limits of $N$-player games (as $N\to+\infty$) and minimal cost.  Similarly to \cite{cecchin2019concergence}, in \cite{delarue2020selection} the authors write down the one dimensional scalar conservation law, which plays the role of the master equation, and value function selection can be done via the entropy solution of this PDE. In a similar spirit, in \cite{bayraktar2020non} the authors also rely on the entropy solution of a conservation law to select equilibria for a finite state space MFG model. For more complicated models than linear-quadratic, in general finite dimensional common noise is not sufficient to restore uniqueness, as is demonstrated in \cite{delarue2019restoring}, where a suitable infinite dimensional noise had to be added to restore uniqueness. In the same spirit, in the case of finite state spaces, the uniqueness restoring can be achieved by adding an appropriate noise of Wright--Fisher type (cf. \cite{bayraktar2021finite}). The program initiated in this work was taken further in \cite{cecchin2022selection}, where the authors consider a finite state space potential MFG. The potential (and finite state space) structure allowed them to perform a sort of vanishing viscosity procedure, which beside the selection principle, led to the definition of a suitable notion of weak entropy solution to the associated master equation. Going even further, a breakthrough came in the very recent work \cite{cecchin2022weak} of the same authors, where they focus on showing that the master equation associated to a class of potential MFGs (whose state space is $\bb{T}^d$) has a unique suitably defined entropy solution. The main idea of this work is to generalize the theory of Kru\v{z}kov (cf. \cite{kruvzkov1970first}) to the infinite dimensional space $\sP(\bb{T}^d)$. The potential structure allows the authors to study a more standard Hamilton--Jacobi--Bellman equation above the master equation, which by differentiating in a suitable sense in the direction of the measure variable formally gives the master equation. There is a small caveat however in this approach, as by this formal differentiation procedure a correction term has to be added to the master equation, and therefore the notion of entropy solutions is given in fact for this `modified' version of the master equation.  The presence of a non-degenerate idiosyncratic noise seemed to be essential in this work (so that the underlying MFG system can have classical solutions). 

{\bf The summary of our contributions.}

For an arbitrary long time horizon $T>0$ and $m\in\sP_2(\R^d)$, the MFG system and the associated master equation that we study in this paper have the form
\begin{equation} \label{intro:MFG}
	\left\{
	\begin{array}{ll}	
		\partial_t v + {H(x,D_x v)} = 0, & {\rm{in}}\ (0,T)\times\R^d,\\[3pt]
		\partial_t \mu + \nabla\cdot \del{{D_p H(x,D_x v)}\mu}  = 0& {\rm{in}}\ (0,T)\times\R^d,\\[3pt]
		v(0,x) = f(x,\mu_0),\ \ \mu_T=m, & {\rm{in}}\ \R^d
	\end{array}
	\right.\tag{MFG}
\end{equation}
and
\begin{equation}\label{intro:master}
		\left\{
		\begin{array}{lll}
	\partial_t u(t,x,m) &+ {H\del{x,D_x u(t,x,m)}}\\[3pt] %-\varepsilon \Delta_x u(t,x,m)\\[3pt]
	&+\displaystyle \int_{\R^d} D_m u(t,x,m)(y) \cdot {D_p H\del{y,D_x u\del{t,y,m}}}\dif m(y) = 0,  &{\rm{in}}\ (0,T)\times\R^d\times\sP_2(\R^d),\\
		u(0,x,m) & = f\del{x,m},  &{\rm{in}}\ \R^d\times\sP_2(\R^d),
		\end{array}
		\right.\tag{ME}
	\end{equation} 
respectively, where the Hamiltonian $H$ and the initial cost $f$ are given. {We would like to note that in our study the Hamiltonian $H$ is taken to be measure independent, hence the interaction among the agents is happens only through $f$. Despite this restriction, we reveal some new nontrivial features of these MFGs. We believe that these initial investigations could lead to the study of more complex models.}

As our main contribution in this paper, for a given class of data $H$ and $f$, we derive a nonlinear transport equation, set on the space of probability measures, written for an {\it auxiliary variable} that we associate to the MFG. The MFG and the master equation can be studied directly by relying on the properties of this transport equation.

The general philosophy behind the construction of the transport equation can be described as follows. Imagine that the coupling function $f:\R^d\times\sP_2(\R^d)\to\R$ in the MFG problem has the specific structure 
$$f(x,m)=g(x,\sigma_0(m)),$$ 
$g:\R^d\times\mathcal{X}\to\R$ and $\sigma_0:\sP_2(\R^d)\to\mathcal{X}$ being given, where the range of $\sigma_0$ is a given Hilbert space $\mathcal{X}$. The motivation behind such a consideration is that in many applications the data functions in a MFG involve only a specific dependence on the measure, involving moments, specific convolutions, etc.~of the measure variable, which are often finite dimensional quantities. Take for instance $f(x,m) := g\left(x,\int_{\R^d}h(y)\dif{m}(y)\right)$ (for a given function $h:\R^d\to\R^n$, representing for instance some generalized moment of $m$), in which case $\mathcal{X}=\R^n$, as $\sigma_0(m)=\int_{\R^d}h(y)\dif{m}(y)$. {Factorizations of similar flavor have recently been proposed in \cite{lasry2022dimension} in the context of dimension reduction techniques for mean field games.}

The next key element is to rewrite the fixed point formulation of the MFG in terms of a new auxiliary variable $\sigma$. Roughly speaking, for $m\in\sP_2(\R^d)$, we define $\sigma(t,m)$ as the solution for the equation
\begin{equation}\label{intro:FP}
\sigma=\sigma_0(\mu^\sigma_t),\tag{FixedP}
\end{equation}
where $\mu^\sigma_t$ is the solution to the continuity equation, where the initial datum for the HJB equation is given by $g(\cdot,\sigma)$. We notice that \eqref{intro:FP} is a fixed point equation in the space $\mathcal{X}$.

By doing so, we can reveal some new hidden structures of the MFG. {\it First}, as sufficient conditions for the unique solvability of the fixed point problem \eqref{intro:FP}, we are able to identify new monotonicity conditions on the data functions, which are in line with the structure of the specific space $\mathcal{X}$. Indeed, the new monotonicity condition can be expressed via the monotonicity of  the operators $\Sigma_t$ (or $-\Sigma_t$), where $\Sigma_t:\mathcal{X}\to\mathcal{X}$ is simply given by
$$
\Sigma_t(\sigma) = \sigma - \sigma_0(\mu^\sigma_t), \ t\in[0,T].
$$
Remarkably, these new monotonicity conditions are in general not covered by the well-understood LL and displacement monotonicity conditions from the literature (nor even by the so-called anti-monotone version of these, considered in \cite{mou2022mean}), and moreover, they are in general in dichotomy with these (cf.~\cite{graber2022unique}). {\it Second}, we can write a nonlinear transport equation for $\sigma$, which can  then be used to study the associated master equation. Indeed, the formula \eqref{intro:FP} suggests that $\sigma(t,m)$ is formally defined via a sort of d'Alembert formula, where $t\mapsto\mu_t$ stands for the characteristic curve. Thus, the transport equation is  given formally by
\begin{equation} \label{intro:transport}
%		\begin{split}
	\left\{
	\begin{array}{lll}
			\partial_t \sigma(t,m) %&- \varepsilon\displaystyle\int_{\R^d} \nabla_\xi \cdot D_m \sigma(t,m)(\xi)\dif m(\xi) & \\
			 &+ \displaystyle\int_{\R^d} D_m \sigma(t,m)(\xi) \cdot D_p H\del{\xi,D_x v\del{t,\xi,\sigma(t,m)}}\dif m(\xi)= 0, & {\rm{in}}\ (0,T)\times\sP_2(\R^d),\\
			& \sigma(0,m) = \sigma_0(m), &{\rm{in}}\ \sP_2(\R^d).
	\end{array}
	\right.\tag{TE}
%		\end{split}
	\end{equation}
In our study we choose $\mathcal{X}$ to be a finite dimensional space, most often $\mathcal{X}=\R$, in which case \eqref{intro:transport} is a scalar equation. Here by $D_m\sigma$ we denote the Wasserstein gradient of $\sigma$. 
	The connection between the solution $u$ to \eqref{intro:master}, the solution $v = v(t,x,\sigma(t,m))$ to the HJB equation from \eqref{intro:MFG} (solved on the time interval $(0,t)$ when the HJB equation is initialized at $g(x,\sigma(t,m))$) and the solution $\sigma = \sigma(t,m)$ to \eqref{intro:transport} is given by
	$$
	u(t,x,m)=v(t,x,\sigma(t,m)).
	$$
	At this point, the study of the transport equation \eqref{intro:transport} is nontrivial because it is essentially coupled to the system \eqref{intro:MFG}, which is itself a coupled system.
	
	The main goal of the current manuscript is to initiate the first steps in the study of \eqref{intro:transport} and its links to the corresponding master equation and selection principles.
{We will} assume that the Hamiltonian is independent of the measure variable, i.e.~$H(x,p,m)=H(x,p)$, {and so} the solution $v = v(t,x,\sigma)$ to the HJB equation depends only on the parameter $\sigma \in \s{X}$, and thus \eqref{intro:transport} is a self-contained transport equation.
We will impose this assumption throughout this paper, leaving more general cases for future research. 
To convince the reader that such an assumption is not overly simplistic, let us consider a simple coordination game in which players must decide how much to shift their initial state in light of what they predict will be the \emph{final} distribution of states after everyone has shifted
(for example, the state variable could represent a position in some sort of ideological space, and the final distribution would then correspond to the popularity of each ideological position, which individuals must take into account before choosing a final position).
Thus each individual player might seek to minimize
\begin{equation*}
	\frac{\del{x-y}^2}{2t} + \sigma_0(m) y
\end{equation*}
over all possible final positions $y$, where $x$ is the player's initial position, $t$ is a given time horizon, and $m$ is the final anticipated distribution of states.
The quadratic cost represents a personal cost to changing one's position, and $\sigma_0(m)$ is a function that uses the final distribution to determine the marginal cost of increasing one's position.
This leads to a mean field game whose solution turns out to have a formal connection to a classical conservation law; see Section \ref{sec:entropy}.
More generally, when the final cost is nonlinear with respect to the state variable, the game will have a more complex structure, which is contained in the transport equation \eqref{intro:transport} and cannot be reduced to a classical equation on finite dimensional space. For this reason, it would not be an exaggeration to say that we are quite far from having a complete understanding of mean field games and the corresponding master equations even in the case when $H(x,p,m)=H(x,p)$. Our aim is to use this restrictive framework to reveal important underlying structure that we expect to be useful for understanding the selection of equilibria in mean field games with continuous state space.

Before elaborating further, we emphasize that in our study we require neither the regularizing effect of a non-degenerate idiosyncratic noise, nor a potential structure for the MFG. In particular, there is no underlying differentiation involved to relate the master equation to the transport equation for $\sigma$. It is also worth noticing that \eqref{intro:transport} is a {\it nonlinear} transport equation, so structurally it is completely different from the linear equations considered in \cite{buckdahn2017mean, chassagneux2022weak}.

%Looking at the master equation through this new transport equation makes it evident that the former has the nature of a transport equation, rather than an HJB equation linked to optimal control problems. Such observations, although in a completely different setting, were made also in \cite{cecchin2022selection,cecchin2022weak}. 

The new monotonicity condition imposed on $\Sigma_t$ (or $-\Sigma_t$) will ensure the global in time well-posedness of \eqref{intro:transport} in the classical sense. In turn, this ensures the well-posedness of \eqref{intro:master} in the classical sense, in the case of data functions that are in general outside of the LL or displacement monotone regime, in the case of deterministic problems. %{We notice that the intensity of a non-degenerate idiosyncratic noise can compensate the lack of monotonicity of $\Sigma_t$; hence, it also restores uniqueness in the case of MFG. In a similar spirit, the restoration of uniqueness via idiosyncratic noise has been recently explored in \cite{cirant2021long}. 
		
Although the nonlinear transport equation character of the master equation is evident to experts (cf.  \cite{cecchin2022selection,cecchin2022weak}), the new transport equation \eqref{intro:transport} and its link to \eqref{intro:master} have a clear advantage for a rich class of models, which we explore in this paper. In certain cases, \eqref{intro:transport} reduces to a classical conservation law, which allows us to apply the theory of entropy solutions in a straightforward manner. By contrast, the standard master equation \eqref{intro:master} does not have such a clean structure, even in these restrictive cases. So, in some regimes where classical solutions to \eqref{intro:master} cease to exist (which is typically the case when monotonicity conditions on the data are violated), one can propose suitable notions of weak solutions to the master equation, via weak solutions to \eqref{intro:transport}.

%The profound link between \eqref{intro:master} and \eqref{intro:transport} has yet another great advantage that we explore in this paper. In regimes where classical solutions to \eqref{intro:master} cease to exist (which is typically the case when monotonicity conditions on the data are violated), one can propose suitable notions of weak solutions to the master equation, via weak solutions to \eqref{intro:transport}. 
{Instead of} studying such weak solutions to this equation in great generality, {our objective in the last part of this paper} is to understand simple scenarios. So, as our initial investigation, in this manuscript we present a case study {(leaving more general cases to future study)}: when $H(x,p,\mu)=H(p)$, \eqref{intro:transport} can be seen as a deterministic scalar conservation law, for which the suitable notion of weak solution is that of weak entropy solution. Furthermore, if the initial cost function has the form
$$g(x,\sigma_0(m))=x\cdot \bar f(\sigma_0(m)),$$ 
for some $\bar f:\mathcal{X}\to\R^d$, \eqref{intro:transport} can be identified with a finite dimensional scalar conservation law written for the mean of probability measures. This dimension reduction (see also \cite{lasry2022dimension} for dimension reduction techniques in the study of MFG) allows us in particular to consider initial data $\sigma_0$ that are even discontinuous. In this case study, we have a full description of the links between Nash equilibria of the MFG and entropy solutions to \eqref{intro:transport}. We construct some simple, yet sobering, examples where Nash equilibria of the MFG cannot be selected by entropy solutions of the conservation law. First, we provide examples where Nash equilibria do not even exist, despite the fact that the conservation law has a unique entropy solution. Even more strikingly, we show that even when Nash equilibria do exist, it could be that none of them is given by the entropy solution to the transport equation.
	
The structure of the rest of the paper is as follows. In Section \ref{sec:simple game}, after introducing some preliminary notions and notations, we describe the new fixed point formulation of deterministic MFG {(involving measure independent Hamiltonians)} in terms of the auxiliary quantity $\sigma$, we present the construction of our main transport equation associated to the MFG and explain the connection between the transport equation and the master equation. In Section \ref{sec:classical}, still in the deterministic setting, we show that our newly proposed monotonicity condition yields the existence of a global in time classical solution to the transport equation, which in turn implies the global in time well-posedness of the master equation. %In Section \ref{sec:restoring}, we demonstrate that uniqueness of MFG equilibria can be restored by adding a Brownian idiosyncratic noise with large enough intensity. More precisely, the intensity of the idiosyncratic noise enters in the definition of the new monotonicity condition. In this regime, we show that both the transport and the master equations are globally well-posed. 
Finally, in Section \ref{sec:entropy} we construct entropy solutions to the transport equation in a specific regime, when this reduces to a finite dimensional scalar conservation law. We end the manuscript with the description of several concrete examples, where we demonstrate how the corresponding MFG Nash equilibria are or are not linked to the entropy solutions that we have constructed.

	\section{Preliminaries and the setup for the model mean field game} \label{sec:simple game}
	
	\subsection{Preliminaries}
	
Let us introduce some notations and preliminaries that will be used throughout the paper. 

We denote by $p^1,p^2:\bb{R}^d \times \bb{R}^d \to \bb{R}^d$ the canonical projections given by $p^1(x,y) = x$ and $p^2(x,y) = y$, respectively.
	If $\s{X}$ and $\s{Y}$ are topological spaces, then for a Borel measurable map $f:\mathcal{X} \to \mathcal{Y}$ and a measure $m$ supported on $\s{X}$, we denote by $f_\sharp m$ the push-forward measure supported on $\s{Y}$ given by the relation $(f_\sharp m)(E) = m\del{f^{-1}(E)}$.	
	For $p\ge 1$, we use the notation $\sP_p(\R^d)$ to denote the space of nonnegative Borel probability measures supported in $\R^d$ with finite $p-$order moments. On $\sP_p(\R^d)$ we define the standard $p-$Wasserstein distance $W_p:\sP_p(\R^d)\times\sP_p(\R^d)\to[0,+\infty)$ as,
	$$W_p(\mu,\nu):=\inf\left\{\int_{\R^d\times\R^d}|x-y|^p\dif{\gamma}:\ \ \gamma\in\sP_p(\R^d\times\R^d),(p^1)_\sharp\gamma=\mu,(p^1)_\sharp\gamma=\mu\right\}^{\frac1p}.$$
Classical results imply that there exists at least one optimizer $\gamma$ in the previous problem. We denote by $\Pi_o(\mu,\nu)$ the set of all optimal plans $\gamma$.
	
	Let $(\Omega,\s{A},\bb{P})$ be an atomless probability space. We use the notation $\bb{H}:=L^2(\Omega;\R^d).$ It is a well-known result that if $\bb{P}$ has no atoms, then for each $m \in \sP_2(\bb{R}^d)$ there exists $X \in \bb{H}$ such that $X_\sharp \bb{P} = m$.  In this case, $m$ is the law of the random variable $X$.
	
{\bf Derivatives for functions defined on the space of measures.} Using the terminology from \cite{ambrosio2008gradient}(see also \cite[Chapter 5]{carmona2018probabilistic}), we say that a function $U:\sP_2(\R^d)\to\R$ has a {\it Wasserstein gradient} at $m\in\sP_2(\R^d)$, if there exists $D_m U(m,\cdot)\in \overline{\nabla C_c^\infty(\R^d)}^{L^2_m}$ (the closure of gradients of $C_c^\infty(\R^d)$ function in $L^2_m(\R^d;\R^d)$) such that  for all $m'\in\sP_2(\R^d)$ in any small neighborhood of $m$ we have the first order Taylor expansion
	$$
	U(m') = U(m) +\iint_{\R^d\times\R^d} D_mU(m,x)\cdot(y-x)\dif \gamma(x,y) + o(W_2(m,m')),\ \forall\gamma\in\Pi_o(m,m').
	$$
We say that $U$ is {\it differentiable on $\sP_2(\R^d)$} if its Wasserstein gradient exists at any point.
	
For $U:\sP_2(\R^d)\to\R$, we can define its `lift' $\tilde U:\bb{H}\to\R$ by $\tilde U(X):=U(X_\sharp\bb{P})$. By the results from \cite{gangbo2019differentiability} and \cite[Chapter 5]{carmona2018probabilistic} (cf. \cite{lions07}), $U$ is differentiable at $m$, if and only if $\tilde U$ is Fr\'echet differentiable at $X$ for any $X\in\bb{H}$, such that $X_\sharp\bb{P}=m$. In this case we have the decomposition 
$$
D\tilde U(X) = D_m U(m,\cdot)\circ X\ \ {\rm{in}}\ \bb{H}, \ \ \forall X\in\bb{H}:\ X_\sharp\bb{P}=m,
$$
where $D\tilde U(X)\in\bb{H}$ stands for the Fr\'echet derivative of $\tilde U$ at $X$.

Based on \cite[Chapter 5]{carmona2018probabilistic}, we say that $U$ is {\it fully $\s{C}^1$ on $\sP_2(\R^d)$} if it continuous, $D_mU$ exists at any point $m\in\sP_2(\R^d)$ and this has a jointly continuous extension to $\sP_2(\R^d)\times\R^d$. In this case, we still continue to denote this extension by $D_m U$. We denote the space of fully $\s{C}^1$ functions by $\s{C}^1(\sP_2(\R^d)).$

We say that $U:\sP_2(\R^d)\to\R$ has a {\it linear derivative at $m$}, if there exists a continuous function $\frac{\delta U}{\delta m}(m,\cdot):\R^d\to\R$ {with at most quadratic growth at infinity}, such that 
$$
\lim_{\varepsilon\to 0}\frac{U(m+\varepsilon{(\bar m-m)})-U(m)}{\varepsilon}=\int_{\R^d}\frac{\delta U}{\delta m}(m,y)\dif (\bar m - m)(y),
$$
for any $\bar m\in\sP_2(\R^d)$. It is well-known (see \cite[Chapter 5]{carmona2018probabilistic}) that if $\frac{\delta U}{\delta m}(m,\cdot)$ exists, and is differentiable in the second variable, with $D_y\frac{\delta U}{\delta m}(m,\cdot)\in L^2_m(\R^d;\R^d)$, then $U$ has a Wasserstein gradient at $m$ and we have the correspondence $D_m U(m,y)=D_y\frac{\delta U}{\delta m}(m,y)$ for any $y\in\spt(m).$

	\subsection{The model mean field game and the formal derivation of the transport equation}\label{subsec:derivation}
	
	We will consider the following mean field game. Let $T>0$ be a given time horizon.
	Let $L:\bb{R}^d \times \bb{R}^d \to \bb{R},$ {$\sigma_0:\sP_2(\bb{R}^d) \to \s{X},$ and $g:\bb{R} \times \s{X} \to \bb{R}$ be continuous functions to be specified later. For ease of presentation, we will restrict ourselves to the scalar case $\s{X} = \bb{R}$, and we will derive a single transport equation. It is easy to generalize the derivation in the case where $\s{X} = \bb{R}^m$, in which case we get a system of $m$ equations, or even when $\s{X}$ is a Hilbert space.
	The results of Section \ref{sec:classical} on classical solutions can be generalized to these cases with little difficulty, bearing in mind the usual notion of monotonicity of vector fields on a Hilbert space.
	By contrast, generalizing the results on entropy solutions in Section \ref{sec:entropy} would not be straightforward, due to the difficulty of studying entropy solutions to systems of conservation laws.}

	For each $t\in(0,T), x \in \bb{R}^d$ and $\sigma \in \bb{R}$, set
	\begin{equation}\label{def:control_prob}
		\begin{split}
			v(t,x,\sigma) &= \min \left\{\int_0^t L\del{x(s),\dot x(s)}\dif s + g(x(0),\sigma) : x(\cdot) \in W^{1,1}\left([0,t];\bb{R}^d\right), \ x(t) = x\right\}, \\ 
			x^*(t,x,\sigma) &= \argmin \cbr{\int_0^t L\del{x(s),\dot x(s)}\dif s + g(x(0),\sigma) : x(\cdot) \in W^{1,1}\del{[0,t];\bb{R}^d}, \ x(t) = x},\\
			x^*(t,x,\sigma,s) &= \cbr{x(s) : x(\cdot) \in x^*(t,x,\sigma)}.
		\end{split}
	\end{equation}
	We will assume properties of $L$ and $g$ that ensure $x^*( t,x,\sigma)$ is always non-empty.
	
	In this mean field game, { we are given a time} $t\in (0,T)$ and an agent distribution $m_t=m\in\sP_2(\R^d)$, initializing the game. Agents make a prediction of the flow of distribution $(m_s)_{s\in[0,t]}$, then a typical individual agent solves the previous optimization problem, where the cost parameter $\sigma$ is determined by plugging the anticipated agent distribution at time $t=0$ into the function $\sigma_0$, i.e. $\sigma = \sigma_0(m_0)$.

	Below we give a precise definition of Nash equilibrium.
	
\begin{definition} \label{def:equilibrium}
	For a given $t>0$ and for $m\in\sP_2(\R^d)$ a given agent distribution at time $t$, a Nash equilibrium is a probability measure $\pi$ on $\bb{R}^d \times \bb{R}^d$ such that $(p^1)_\sharp \pi = m$ and
	\begin{equation} \label{eq:equi def}
		y \in x^*\del{t,x,\sigma_0((p^2)_\sharp \pi),0} \quad \pi-{\rm{a.e.}}~(x,y).
	\end{equation}
\end{definition}	

Definition \ref{def:equilibrium} is essentially the usual one for the Lagrangian formulation of a mean field game.
The only difference is that the running cost does not depend on the distribution of players.
For this reason, we need only define the equilibrium as a coupling between initial and final measures, not as a measure on the space of curves.

Notice that the scalar {$\sigma_0((p^2)_\sharp \pi)$} plays a pivotal role in this definition.
Indeed, an equivalent way of stating the definition of Nash equilibrium is as follows.
For $\sigma \in \bb{R}, m \in \sP(\bb{R}^d),$ and $t > 0$, denote by $\Pi^*(\sigma,m,t)$ the set of all $\pi \in \sP(\bb{R}^d \times \bb{R}^d)$ such that $(p^1)_\sharp \pi = m$ and 
\begin{equation}
	y \in x^*(t,x,\sigma,0) \quad \pi-\text{a.e.}~(x,y).
\end{equation}
Suppose $\sigma$ satisfies
\begin{equation} \label{eq:equi cond}
	\sigma \in \cbr{\sigma_0((p^2)_\sharp \pi) : \pi \in \Pi^*(\sigma,m,t)} =: \Sigma_0(\sigma,t,m).
\end{equation}
Then there exists a $\pi \in \Pi^*(\sigma,m,t)$ such that $\sigma = \sigma_0((p^2)_\sharp \pi)$, which makes $\pi$ an equilibrium.
Conversely, if $\pi$ is an equilibrium, then $\sigma = \sigma_0((p^2)_\sharp \pi)$ satisfies \eqref{eq:equi cond}.
Hence Definition \ref{def:equilibrium} is equivalent to the fixed point problem \eqref{eq:equi cond}.

If $x_0^*(t,x,\sigma)$ is a single-valued function, then so are $\Pi^*(\sigma,m,t)$ and $\Sigma_0(\sigma,t,m)$, namely
\begin{equation} \label{eq:Sigma_0}
	\Pi^*(\sigma,m,t) = x^*(t,\cdot,\sigma,0)_\sharp m, \quad \Sigma_0(\sigma,t,m) = \sigma_0\del{x^*(t,\cdot,\sigma,0)_\sharp m}.
\end{equation}
Then the equilibrium condition reduces to the equation
\begin{equation}\label{eq:equil cond single valued}
	\sigma = \sigma_0\del{x^*(t,\cdot,\sigma,0)_\sharp m}.
\end{equation}

\begin{remark}		
 Another way to define a Nash equilibrium is in terms of random variables.
	%Let $(\Omega,\s{A},\bb{P})$ be an atomless probability space. {We use the notation $\bb{H}:=L^2(\Omega;\R^d).$}
	Define $\tilde \sigma_0 : \bb{H} \to \bb{R}$ by $\tilde \sigma_0(X) := \sigma_0(X_\sharp\bb{P})$.
	%It is a well-known result that if $\bb{P}$ has no atoms, then for each $m \in \sP_2(\bb{R}^d)$, there exists $X \in \bb{H}$ such that $X_\sharp \bb{P} = m$. 
	The Nash equilibrium for such $m$ is given by finding a random variable $Y \in \bb{H}$ such that
	\begin{equation}
		Y \in x^*\del{t,X,\tilde \sigma_0(Y),0},\ \ \bb{P}-{\rm{a.s.}}
	\end{equation}
\end{remark}

	Let us introduce the Hamilton in the standard way by
	$$H(x,p) := \sup_q \del{q \cdot p - L(x,q)}.$$ 
Suppose that the equilibrium problem \eqref{eq:equil cond single valued} has a solution and denote this by $\sigma(t,m)$. Formally, $\sigma(t,m)$ satisfies the following transport equation:
	\begin{equation} \label{eq:transport}
		\left\{
		\begin{array}{ll}
\displaystyle\partial_t \sigma(t,m) + \int_{\R^d} D_m \sigma(t,m)(y) \cdot D_p H\left(y,D_x v\del{t,y,\sigma(t,m)}\right)\dif m(y) = 0, & {\rm{in}}\ (0,T)\times\sP_2(\R^d),\\
\sigma(0,m) = \sigma_0(m), & {\rm{in}}\ \sP_2(\R^d).
		\end{array}
		\right.
	\end{equation}
	To see this, note that for each $\sigma$, $v( t,x,\sigma)$ is the unique viscosity solution of
	\begin{equation} \label{eq:HJ}
	\left\{
	\begin{array}{ll}	
		\partial_t v + H(x,D_x v) = 0, & {\rm{in}}\ (0,T)\times\R^d,\\
		v(0,x) = g(x,\sigma), & {\rm{in}}\ \R^d.
	\end{array}
	\right.
	\end{equation}
	Suppose that $v(\cdot,\cdot,\sigma)$ is differentiable at $(t,x)$, then the corresponding optimal trajectory is unique and given by the solution to
	\begin{equation} 
		\dot x(s) = D_p H\del{x(s),D_x v(s,x(s),\sigma)},\ s\in(0,t), \quad x(t) = x.
	\end{equation}
	Let $m \in \sP_2(\bb{R})$ be given, pick $X \in \bb{H}$ such that $X_\sharp \bb{P} = m$, and let $Y$ be a corresponding equilibrium configuration. Suppose that this equilibrium is unique.
	Then almost surely we have $Y = X(0)$, where $X(s)$ satisfies
	\begin{equation} \label{eq:characteristics}
		\dot X(s) = D_p H\del{X(s),D_x v(s,X(s),\tilde \sigma_0(X(0)))},\ s\in(0,t), \quad X(t) = X.
	\end{equation}
	Moreover, $\sigma(t,m) = \tilde \sigma_0(X(0))$.
	Thus Equation \eqref{eq:characteristics} formally defines characteristics for the flow $\sigma(t,m)$.
	Indeed, if $\tilde \sigma(t,X) := \sigma(t,X_\sharp \bb{P})$, then for any $\tau \in [0,t]$ we have $\tilde \sigma\del{\tau, X(\tau)} = \tilde \sigma_0(X(0))$, because we again have $X(0) \in x_0^*\del{\tau,X(\tau),\tilde \sigma_0(X(0))}$ a.s.
	The PDE whose characteristics are given by \eqref{eq:characteristics}, formally reads as
	\begin{equation} \label{eq:transport lifted}
		\partial_t \tilde \sigma(t,X) + \ip{D_X \tilde\sigma(t,X)}{D_p H\del{X,D_x v(t,X,\tilde\sigma(X,t))}} = 0,
	\end{equation}
	where the inner product is the standard one on $\bb{H}$.
	Projecting down to the Wasserstein space, we arrive at \eqref{eq:transport}.

	Moreover, we define the \emph{master field} $u:(0,T)\times\R^d\times\sP_2(\R^d)\to\R$ by
	\begin{equation}\label{def:master_function}
		u(t,x,m):= v\del{t,x,\sigma(t,m)}.
	\end{equation}
	Formally, $u$ satisfies the master equation
	\begin{equation}\label{eq:master equation}
		\left\{
		\begin{array}{lll}
	\partial_t u(t,x,m) &+ H\del{x,D_x u(t,x,m)}\\ 
		&+\displaystyle \int_{\R^d} D_m u(t,x,m)(y) \cdot D_p H\del{y,D_x u\del{t,y,m}}\dif m(y) = 0,  &{\rm{in}}\ (0,T)\times\R^d\times\sP_2(\R^d),\\
		u(0,x,m) & = g\del{x,\sigma_0(m)},  &{\rm{in}}\ \R^d\times\sP_2(\R^d).
		\end{array}
		\right.
	\end{equation}
	
	To see this, we use the following (formal) identities:
	\begin{equation}\label{eq:correspondence}
		\begin{split}
			\partial_t u(t,x,m) &= \partial_t v\del{t,x,\sigma(t,m)} + \partial_\sigma v\del{t,x,\sigma(t,m)} \partial_t \sigma(t,m),\\
			D_x u(t,x,m) &= D_x v\del{t,x,\sigma(t,m)},\\
			D_m u(t,x,m)(\cdot) &= \partial_\sigma v\del{t,x,\sigma(t,m)}D_m \sigma(t,m)(\cdot).
		\end{split}
	\end{equation}
	Multiplying equation \eqref{eq:transport} by $\partial_\sigma v\del{t,x,\sigma(t,m)}$ and combining with \eqref{eq:HJ}, we obtain indeed \eqref{eq:master equation}.

	\section{Classical solutions in the monotone case}\label{sec:classical}
	
	In this section we will impose the following assumptions.

\begin{assumption}\label{ass:H}
Assume that $H:\R^d\times\R^d\to\R$ is twice continuously differentiable, $D_p H$ is Lipschitz continuous in both variables, and $D^2H$ is uniformly bounded.
\end{assumption}

\begin{assumption}\label{ass:v_C11}
Assume that for each $\sigma \in \range{\sigma_0}$, $v(\cdot,\cdot,\sigma)$ is a \emph{classical} $C^{1,1}_{\rm{loc}}$ solution to the Hamilton--Jacobi equation \eqref{eq:HJ}, such that $\R^d\ni x\mapsto D_x v(t,x,\sigma)$ is Lipschitz continuous, for all $t\in[0,T]$ and $\sigma\in\range{\sigma_0}$, with a Lipschitz constant independent of $t$ and $\sigma$.
\end{assumption}	
\begin{remark}
\begin{enumerate}
\item Let us note that Assumption \eqref{ass:v_C11} includes assumptions on $g(\cdot,\sigma)$ and further assumptions on $H$ in an implicit way. Sufficient assumptions that guarantee the fulfillment of Assumption \ref{ass:v_C11} are the joint convexity of $\R^{2d}\ni(x,w)\mapsto L(x,w)$ (where $L(x,\cdot)=H(x,\cdot)^*$ for all $x\in\R^d$) and the convexity of $\R^d\ni x\mapsto g(x,\sigma)$, together with $D_x^2g(\cdot,\sigma)$ uniformly bounded, for all $\sigma\in\R^d$. By adding idiosyncratic noise with a positive intensity, the required regularity on $v$ would be a consequence of parabolic regularity.
\item As we aim to obtain classical solutions for \eqref{eq:transport} and hence for \eqref{eq:master equation} in the case of deterministic problems, it is necessary to suppose that the Hamilton--Jacobi equation from the corresponding MFG system has a classical solution. For this, it is in general inevitable to impose convexity of the function $g(\cdot,\sigma)$ for all $\sigma$ (see for instance \cite[Theorem 3.1]{graber2022unique}).
\end{enumerate}
\end{remark}	
%	\begin{equation} \label{eq:hj1}
%		\partial_t v + H(x,D_x v) = 0, \quad v(\sigma,x,0) = g(\sigma,x).
%	\end{equation}
	
	We then have for each $t > 0$ a unique optimal flow $[0,t]\ni s \mapsto x^*( t,x,\sigma,s)$ given by solving
	\begin{equation}
		\dot x(s) = D_p H\del{x(s),D_x v\del{s,x(s),\sigma}}, \quad x(t) = x.
	\end{equation}
	Even with $x_0^*(t,\cdot,\sigma,0)$ well-defined (and in particular unique), it can happen that multiple equilibria exist.
	As in the case of conservation laws, we need to impose monotonicity on the initial condition in order to guarantee uniqueness and the propagation of smoothness.
	
	Let us recall the definition of $\Sigma_0$ from \eqref{eq:Sigma_0}, i.e.	
	$$
		\Sigma_0(\sigma,t,m) = \sigma_0\del{x^*(t,\cdot,\sigma,0)_\sharp m}.
	$$
	Note that Equilibrium condition \eqref{eq:equil cond single valued} is the same as
	\begin{equation} \label{eq:equil cond1}
		\sigma = \Sigma_0(\sigma,t,m).
	\end{equation}
	Our key assumption in this section is the following.
	\begin{assumption}\label{eq:monotonicity}
	Assume that $\Sigma_0$ is differentiable with respect to $\sigma$ and there exists $c_0<1$ such that $\partial_\sigma\Sigma_0(\sigma,t,m) \leq c_0,$ for all $(\sigma,t,m)\in\range(\sigma_0)\times[0,T]\times\sP_2(\R^d)$.
	\end{assumption}

\begin{remark}\label{rmk:anti}
\begin{enumerate}
\item Note that Assumption \ref{eq:monotonicity} is a condition that involves the data $H$, $g$ and $\sigma_0$ in an implicit way.
\item We note that in our consideration, one could change the `sign' of the monotonicity condition. Replacing Assumption \ref{eq:monotonicity} with the one that there exists $c_0>1$ such that $\partial_\sigma\Sigma_0(\sigma,t,m) \geq c_0,$ for all $(\sigma,t,m)\in\range(\sigma_0)\times[0,T]\times\sP_2(\R^d)$, would yield the very same well-posedness theory. It is natural to refer to this latter condition as \emph{anti-monotonicity}. However, we remain consistent with the sign of the monotonicity condition (as in Assumption \ref{eq:monotonicity}) throughout the manuscript. 
\item {The hypothesis on $\Sigma_0$ in Assumption \ref{eq:monotonicity} is very much reminiscent of classical monotonicity conditions imposed in the literature for conservation laws. Indeed, if we consider the Cauchy problem associated to a classical scalar conservation law in one space dimension, $\partial_t u(t,x)+\partial_x(G(u(t,x)))=0$, with $u(0,x) = u_0(x)$, the classical monotonicity condition (cf. \cite[Theorem 6.1.1]{dafermos16}) can be imposed as $\partial_x [G'(u_0(x))]\ge 0.$ If this is in place then the Cauchy problem has a globally defined classical solution. Furthermore, as the solution is constant along characteristics and characteristics are straight lines, one can write the implicit (fixed point) equation
$$u(t,x) = u_0(x-G'(u(t,x))).$$
Therefore, the monotonicity described above, can be expressed via the monotonicity of the real valued function $u\mapsto u - u_0(x-G'(u))$, for all $x\in\R$. This map involves the data and the characteristics in an implicit way. So, this is the same philosophy that we use in Assumption \ref{eq:monotonicity}.
}
{\item  It is straightforward to generalize Assumption \ref{eq:monotonicity} to the case where $\sigma$ takes values in $\bb{R}^n$; see the remark following Corollary \ref{cor:trans_master}.}
\end{enumerate}
\end{remark}

\begin{assumption}
$\sigma_0:\sP_2(\R^d)\to\R$ is bounded.
\end{assumption}

	\begin{assumption}\label{ass:sigma_0}
	We assume that $\frac{\delta\sigma_0}{\delta m}$ exists and $\sP_2(\R^d)\times\R^d\ni (m,y)\mapsto \frac{\delta\sigma_0}{\delta m}(m,y)$ is uniformly continuous. We assume furthermore that for any $m\in\sP_2(\R^d)$ the gradient $D_y\frac{\delta\sigma_0}{\delta m}(m,\cdot)$ exists, it is an element of $L^2_m(\R^d;\R^d)$ and it is jointly continuous in $(m,y).$
	\end{assumption}
	
\begin{remark}
We notice that as a consequence of Assumption \ref{ass:sigma_0} we have that $\sigma_0$ is a fully $\s{C}^1$ and $D_y\frac{\delta\sigma_0}{\delta m}:\sP_2(\R^d)\times\R^d\to\R^d$ provides a jointly continuous extension for $D_m\sigma_0$. 
\end{remark}

Before stating and proving our main results in this section, let us pause and discuss more about Assumption \ref{eq:monotonicity}, by providing also examples of data fulfilling it. {For a thorough discussion about properties of this monotonicity condition, we refer to our recent work \cite{graber2022unique}.}

\subsection{Examples satisfying our monotonicity condition} \label{sec:examples}
First, let us observe that by the chain rule, Assumption \ref{eq:monotonicity} can be rewritten as 
\begin{equation}\label{mon:after_chain}
\int_{\R^d}D_m\sigma_0(x^*(t,\cdot,\sigma,0)_\sharp m)(x^*(t,x,\sigma,0))\cdot\partial_\sigma x^*(t,x,\sigma,0)\dif{m}(x)\le c_0,
\end{equation}
for all $(\sigma,t,m)\in \range(\sigma_0)\times[0,T]\times\sP_2(\R^d)$.

Let $H:\R^d\times\R^d\to\R$ depend only on the momentum variable, i.e. $H(x,p)\equiv H(p)$, let  $\vphi,\psi:\R^d\to\R$ and $G:\R\to\R$ be given, and  let us consider $g:\R^d\times\sP_2(\R^d)\to\R$ defined by
\begin{equation}\label{ex:g}
g(x,m)=\vphi(x)G\left(\int_{\R^d}\psi(y)\dif{m}(y)\right).
\end{equation}
With this choice of initial data, one can observe that there are at least two quite natural choices for $\sigma_0$. We discuss these two cases separately.

{\it Case 1.} $\sigma_0:\sP_2(\R^d)\to\R$ is given by $\sigma_0(m):= \int_{\R^d}\psi(y)\dif{m}(y)$. 

Rewriting $g$ using the variable $\sigma\in\R$, we have $g(x,\sigma)=\vphi(x)G(\sigma),$ and so using the Hopf--Lax formula, we can define $x^*(t,x,\sigma,0)$ as the unique solution of 
$$
\min_{y\in\R^d}\left\{tH^*\left(\frac{x-y}{t}\right)+g(y,\sigma)\right\}.
$$
By differentiation, we find that $x^*(t,x,\sigma,0)$ satisfies the implicit equation
$$
x^*(t,x,\sigma,0)+t DH(D_xg(x^*(t,x,\sigma,0),\sigma)) = x,
$$
from which, with the notation 
\begin{equation}\label{def:M}
M(t,z,\sigma):=- t \left(I_d+t D^2H(D_xg(z,\sigma))D^2_{xx}g(z,\sigma)\right)^{-1}D^2H(D_xg(z,\sigma)),
\end{equation}
one obtains
\begin{align*}
\partial_\sigma x^*(t,x,\sigma,0) &= M(t,x^*(t,x,\sigma,0),\sigma) \partial_\sigma D_xg(x^*(t,x,\sigma,0),\sigma)\\
& = G'(\sigma)M(t,x^*(t,x,\sigma,0),\sigma)D\vphi(x^*(t,x,\sigma,0)).
\end{align*}
Let us underline that it is not a loss of generality to assume the invertibility of the matrix 
$$I_d+t D^2H(D_xg(z,\sigma))D^2_{xx}g(z,\sigma)$$ 
in the definition \eqref{def:M}. Indeed, this is strongly related to the solvability of the associated HJB equation in the classical sense, as we can see in \cite[Theorem 1.5.3]{cannarsa2004semiconcave}. In particular, if $H$ and $g(\cdot,\sigma)$ are supposed to be convex, this property holds true for all $t>0$ (cf. \cite[Corollary 1.5.5]{cannarsa2004semiconcave}).

We also have 
$$
D_m\sigma_0(m,y)=D\psi(y),
$$
and so, \eqref{mon:after_chain}, with the notation $m_{t,\sigma}:=x^*(t,\cdot,\sigma,0)_\sharp m$, can be written as 
\begin{align}\label{eq:case1}
\nonumber&\int_{\R^d}D\psi(x^*(t,x,\sigma,0))\cdot G'(\sigma)M(t,x^*(t,x,\sigma,0),\sigma)D\vphi(x^*(t,x,\sigma,0))\dif{m}(x)\\
&=\int_{\R^d}D\psi(y)\cdot G'(\sigma)M(t,y,\sigma)D\vphi(y)\dif{m}_{t,\sigma}(y)\le c_0,
\end{align}
for all $m\in\sP_2(\R^d)$, $t\in [0,T]$ and $\sigma\in\range(\sigma_0)$.

{\it Case 2.} $\sigma_0:\sP_2(\R^d)\to\R$ is given by $\sigma_0(m):= G\left(\int_{\R^d}\psi(y)\dif{m}(y)\right)$.

In this case, using the $\sigma$ variable, we have $g(x,\sigma)=\vphi(x)\sigma.$ A similar computation as in the previous case yields
$$
\partial_\sigma x^*(t,x,\sigma,0) = M(t,x^*(t,x,\sigma,0),\sigma)D\vphi(x^*(t,x,\sigma,0)).
$$
Furthermore, 
$$
D_m\sigma_0(m,x) = G'\left(\int_{\R^d}\psi(y)\dif{m}(y)\right) D\psi(x).
$$
Therefore, \eqref{mon:after_chain} reads as 
\begin{align}\label{eq:case2}
\nonumber&\int_{\R^d}G'\left(\int_{\R^d}\psi(y)\dif{m}_{t,\sigma}(y)\right) D\psi(x^*(t,x,\sigma,0))\cdot M(t,x^*(t,x,\sigma,0),\sigma)D\vphi(x^*(t,x,\sigma,0))\dif{m}(x)\\
&=\int_{\R^d}G'\left(\int_{\R^d}\psi(y)\dif{m}_{t,\sigma}(y)\right)D\psi(y)\cdot M(t,y,\sigma)D\vphi(y)\dif{m}_{t,\sigma}(y)\le c_0
\end{align}
for all $m\in\sP_2(\R^d)$, $t\in [0,T]$ and $\sigma\in\range(\sigma_0)$.

These computations allow us to formulate the following sufficient condition on our data, yielding the monotonicity hypotheses.
\begin{lemma}\label{lem:ex}
Let $g$ be given as in \eqref{ex:g}.
\begin{enumerate}
\item Let $\sigma_0:\sP_2(\R^d)\to\R$ be given by $\sigma_0(m):= \int_{\R^d}\psi(y)\dif{m}(y)$. Suppose that $\vphi,\psi:\R^d\to\R$ are twice continuously differentiable, $\psi$ and $D\psi$ are uniformly bounded. Suppose that $G:\R\to\R$ is continuously differentiable and $G(s)\ge 0$ for all $s\in[\min\psi,\max\psi]=\range(\sigma_0)$. Suppose that $\vphi$ is convex, $D\vphi$ is Lipschitz continuous and $H(x,p)\equiv H(p)$ satisfies Assumption \ref{ass:H}. 
\begin{enumerate}
\item[(i)] Suppose that there exists $c_0<1$ such that 
\begin{equation}\label{ineq:pointwise}
D\psi(y)\cdot G'(\sigma)M(t,y,\sigma)D\vphi(y)\le c_0
\end{equation} 
for all $(t,y,\sigma)\in [0,T]\times\R^d\times \range(\sigma_0)$. Then our data satisfy all the assumptions on this section, and in particular the monotonicity condition imposed in Assumption \ref{eq:monotonicity}.
\item[(ii)] If instead, there exists $c_0>1$ such that $D\psi(y)\cdot G'(\sigma)M(t,y,\sigma)D\vphi(y)\ge c_0$ for all $(t,y,\sigma)\in [0,T]\times\R^d\times \range(\sigma_0)$, then our data satisfy the anti-monotonicity condition described in Remark \ref{rmk:anti}(2).
\end{enumerate}
\item Let $\sigma_0:\sP_2(\R^d)\to\R$ be given by $\sigma_0(m):= G\left(\int_{\R^d}\psi(y)\dif{m}(y)\right)$. Suppose that $G:\R\to\R$ is continuously differentiable, nonnegative, bounded with bounded derivative. Suppose that $\psi:\R^d\to\R$ is continuously differentiable with $D\psi$ having at most linear growth at infinity. Suppose that $\vphi$ and $H$ are as in point (1) of this lemma.
\begin{enumerate}
\item[(i)] Suppose that there exists $c_0<1$ such that $D\psi(y)\cdot G'(s)M(t,y,\sigma)D\vphi(y)\le c_0$ for all $(t,y,\sigma,s)\in [0,T]\times\R^d\times \range(\sigma_0)\times\R$, where the matrix $M$ is defined in \eqref{def:M}. Then our data satisfy all the assumptions on this section, and in particular the monotonicity condition imposed in Assumption \ref{eq:monotonicity}.
\item[(ii)] If instead, there exists $c_0>1$ such that $D\psi(y)\cdot G'(s)M(t,y,\sigma)D\vphi(y)\ge c_0$ for all $(t,y,\sigma,s)\in [0,T]\times\R^d\times \range(\sigma_0)\times\R$, then our data satisfy the anti-monotonicity condition described in Remark \ref{rmk:anti}(2).
\end{enumerate}
\end{enumerate}
\end{lemma}

\begin{proof}
The proof of this lemma is immediate by the previous computations, and in particular the monotonicity assumption is a consequence of \eqref{eq:case1} and \eqref{eq:case2}.
\end{proof}

\begin{remark}
\begin{enumerate}
\item In the statement of the above lemma, in inequality \eqref{ineq:pointwise} (and similarly in all the similar inequalities) the constant $c_0$ a priori might depend on the time horizon $T$. However, clearly, this monotonicity condition in general is in fact independent of time. Indeed, as we have discussed above, by the convexity of $H$ and $g(\cdot,\sigma)$,  \cite[Corollary 1.5.5]{cannarsa2004semiconcave} yields the invertibility of the matrix $I_d+t D^2H(D_xg(z,\sigma))D^2_{xx}g(z,\sigma),$
independently of $t>0$. Moreover, since $D^2H$ and $D^2_{xx}g(\cdot,\sigma)$ are bounded, for large $t>0$, $M(t,\cdot,\sigma)$ is of constant order (for large $t$) for all $\sigma$.
\item As for a concrete example, consider $H(p)=\frac12|p|^2$, $\vphi(x)=\frac12|x|^2$, $\psi$ bounded, let $G$ be bounded below by a positive constant and increasing on the interval $[\min\psi,\max\psi]$, and $D\psi(y)\cdot y\ge 0$ (for instance, $\psi(y)=\arctan(|y|^2)$ satisfies these conditions). In this case, the assumptions of Lemma \ref{lem:ex}(1)(i) are fulfilled, and hence all the assumptions of this section are fulfilled.
\item In general the class of examples considered in \eqref{ex:g} do not possess a potential structure. Indeed, the function $g$ is derived from a potential if and only if $\vphi$ and $\psi$ are proportional. Moreover, as is detailed in \cite{graber2022unique}, this class of examples is in general in dichotomy with both the LL monotonicity and displacement monotonicity (and the related anti-monotonicity conditions as well). 
%\item Let us consider a similar setting with $H(p)=\frac12|p|^2$, $\vphi(x)=\frac12|x|^2$, and so
%\begin{align*}
%D\psi(y)\cdot G'(\sigma)M(t,y,\sigma)D\vphi(y) = -\frac{t}{1+tG(\sigma)}G'(\sigma)D\psi\cdot y.
%\end{align*}
%[@Jameson: in point (3) I wanted to construct an example for the anti-monotone case, by simply changing the properties of $G$ from the previous example. However, it seems to me that this will be impossible, if the time is short. This seems to me quite interesting. I am wondering now whether should one really think that one needs to come up with a more sophisticated example in the anti-monotone case, just as in [Mou-Zhang]? I will try to think about the anti-monotone case a bit more.]
\end{enumerate}
\end{remark}

\subsection{Main result}
	
	\begin{theorem}\label{thm:transport_deterministic} 
	Let the data satisfy Assumptions \ref{ass:H}, \ref{ass:v_C11}, \ref{eq:monotonicity} and \ref{ass:sigma_0}.
%		Let $H$ be $\s{C}^2$ with bounded second derivative.
%		Assume that for each $\sigma \in \bb{R}$, equation \eqref{eq:HJ} admits a unique $\s{C}^{1,1}$ solution with optimal flow given by $x_0^*( t,x,\sigma,s)$. 
%		Let $\sigma_0$ be a fully $\s{C}^1$ function on $\sP_2(\R^d)$ {[need to recall this in the `Preliminaries' section above]}.
%		Assume that $\Sigma_0$ defined above is differentiable with respect to $\sigma$ and that Equation \eqref{eq:monotonicity} holds {[Can one spell Equation with a lower-case letter?]}.
		Then Equation \eqref{eq:transport} has a unique classical solution $\sigma$, and $\sigma(t,m)$ is the unique mean field Nash equilibrium for every $m \in \sP_2(\R^d)$ and $t > 0$.
	\end{theorem}

	\begin{proof}
		\firststep For any $m \in \sP_2(\R^d)$ and $t > 0$, Equation \eqref{eq:monotonicity} shows that
		\begin{equation}
			\sigma \mapsto \sigma - \Sigma_0(\sigma,t,m)
		\end{equation}
		has a derivative bounded below by $1-c_0 > 0$ and is therefore an invertible function on $\bb{R}$.
		Thus, we can define $\sigma(t,m)$ to be the unique solution to $\sigma = \Sigma_0(\sigma,t,m)$, which is precisely the equilibrium condition \eqref{eq:equil cond1}.
		
		\nextstep Let us show that $\Sigma_0(\sigma,t,m)$ is Lipschitz continuous and differentiable with respect to the $m$ variable.
		Using the regularity of $\sigma_0$, we have
		\begin{equation}
			\begin{split}
				\Sigma_0(\sigma,t,\tilde m) - \Sigma_0(\sigma,t,m)
			&= \int_0^1 \int_{\bb{R}^d} \vd{\sigma_0}{m}\del{x^*(t,\cdot,\sigma,0)_\sharp m_\lambda,x}\dif \del{x^*(t,\cdot,\sigma,0)_\sharp \del{\tilde m - m}}(x)\dif \lambda\\
			&= \int_0^1 \int_{\bb{R}^d} \vd{\sigma_0}{m}\del{x^*(t,\cdot,\sigma,0)_\sharp m_\lambda,x^*( t,x,\sigma,0)}\dif  \del{\tilde m - m}(x)	\dif \lambda		
			\end{split}
		\end{equation}
		where $m_\lambda = \lambda \tilde m + (1-\lambda)m$.
		It follows that $\Sigma_0$ has a linear derivative with respect to $m$ given by
		\begin{equation}
			\vd{\Sigma_0}{m}(\sigma,t,m,x) = \vd{\sigma_0}{m}\del{x^*(t,\cdot,\sigma,0)_\sharp m,x^*( t,x,\sigma,0)}.
		\end{equation}
		This is differentiable with respect to $x$ with
		\begin{equation}
			D_x\vd{\Sigma_0}{m}(\sigma,t,m,x) = D_m\sigma_0\del{x^*(t,\cdot,\sigma,0)_\sharp m,x^*( t,x,\sigma,0)}D_x x^*( t,x,\sigma,0),
		\end{equation}
		where $D_x x^*( t,x,\sigma,0)$ is well-defined and bounded, by the regularity of $H$ and $v$.
		Appealing to \cite[Proposition 5.48]{carmona2018probabilistic}, we see that $\Sigma_0$ has a Wasserstein derivative $D_m \Sigma_0 = D_x\vd{\Sigma_0}{m}$ with respect to $m$, which is also bounded.
		It immediately follows that it is Lipschitz with respect to $m$.
		
		\nextstep We now establish that $\sigma(t,m)$ is differentiable with respect to $m$.
		We claim that
		\begin{equation} \label{eq:sigma lin der}
			\vd{\sigma}{m}(t,m,x) = \del{1 - \partial_\sigma \Sigma_0\del{\sigma(t,m),t,m}}^{-1}\vd{\sigma_0}{m}\del{m,x^*\del{t,x,\sigma(t,m),0}}.
		\end{equation}
		For this, start by defining
		\begin{equation}
			\rho(\sigma,t,m) := \sigma - \Sigma_0(\sigma,t,m).
		\end{equation}
		By the previous step and the assumption on $\Sigma_0$, $\rho$ is differentiable in both $\sigma$ and $m$, and we have
		\begin{equation} \label{eq:sigma differences}
			\begin{split}
				\rho\del{\sigma(t,\tilde{m}), t, \tilde{m}} - \rho\del{\sigma(t,m),t,m} 
				&= \int_0^1 \partial_\sigma\rho\del{\sigma_\lambda,t,{m}_\lambda}\del{\sigma(t,\tilde{m}) - \sigma(t,m)} \dif \lambda\\
				&\quad -  \int_0^1\int_{\bb{R}^d} \vd{\Sigma_0}{m}\del{\sigma_\lambda,t,{m}_\lambda,x}\dif(\tilde{m} - m)(x)\dif \lambda
			\end{split}
		\end{equation}
		where $\sigma_\lambda := \lambda\sigma(t,{m}) + (1-\lambda)\sigma(t,\tilde{m})$ and $m_\lambda = \lambda \tilde m + (1-\lambda)m$.
		By definition of the function $\sigma(t,m)$,
		$$
			\rho\del{\sigma(t,\tilde{m}), t,\tilde{m}} = \rho\del{\sigma(t,m),t,m} = 0.
		$$
		Moreover, $\partial_\sigma\rho = 1 - \partial_\sigma\Sigma_0 \leq 1 - c_0 > 0$.
		We deduce from \eqref{eq:sigma differences} that
		\begin{align*}
			\abs{\sigma(t,\tilde{m}) - \sigma(t,m)}
			&\leq \frac{1}{1-c_0}\abs{\int_0^1\int_{\bb{R}^d} \vd{\Sigma_0}{m}\del{\sigma_\lambda,t,{m}_\lambda,x}\dif(\tilde{m} - m)(x)\dif \lambda}
			\leq \frac{\sup \abs{D_m \Sigma_0}}{1-c_0}W_1(\tilde m,m)\\
			&\leq \frac{\sup \abs{D_m \Sigma_0}}{1-c_0}W_2(\tilde m,m),
		\end{align*}
		and thus $\sigma(t,m)$ is $W_2$-Lipschitz continuous with respect to $m$.
		Then \eqref{eq:sigma differences} implies further that
		\begin{equation*}
			\del{1- \partial_\sigma \Sigma_0\del{\sigma(t,m),t,m}}\del{\sigma(\tilde{m},t) - \sigma(t,m)} = \int_{\bb{R}^d} \vd{\Sigma_0}{m}\del{\sigma(t,m),t,m,x}\dif(\tilde{m} - m)(x) 
			+ o\del{W_2(\tilde m,m)},
		\end{equation*}
		from which we deduce that $\sigma$ has a linear derivative with respect to $m$ given by \eqref{eq:sigma lin der}.
		We can see that it is differentiable with respect to $x$ with bounded derivative, from which it follows that $\sigma(t,m)$ has a bounded Wasserstein gradient with respect to $m$.
		
		\nextstep
		Fix $t,m,$ and $\sigma = \sigma(t,m)$.
		Set $\mu(s) = x^*(t,\cdot,\sigma,s,0)_\sharp m$.
		By the flow property, we have
		\begin{equation}
			x^*\del{\tau,x^*\del{ t,x,\sigma,\tau},\sigma,s} = x^*\del{ t,x,\sigma,s} \quad \forall 0 \leq s \leq \tau \leq t.
		\end{equation}
		Thus
		\begin{equation}
			\sigma_0\del{x^*(t,\cdot,\sigma,0)_\sharp m} = \sigma_0\del{x^*(\tau,\cdot,\sigma,0)_\sharp \mu(\tau)} \quad \forall 0 \leq \tau \leq t,
		\end{equation}
		from which we deduce
		\begin{equation}
			\sigma(t,m) = \sigma\del{\tau,\mu(\tau)} \quad \forall 0 \leq \tau \leq t.
		\end{equation}
		Fix $0 < \tau < t$.
		Set $\mu_\lambda = \lambda \mu(\tau) + (1-\lambda) m$.
		We have
		\begin{equation}
			\begin{split}
				\sigma(t,m) - \sigma(\tau,m) &= \sigma(\tau,\mu(\tau)) - \sigma(\tau,m)\\
				&= \int_0^1 \int_{\bb{R}^d} \vd{\sigma}{m}(\tau, \mu_\lambda,x)\dif \del{\mu(\tau) - m}(x)\dif \lambda\\
				&= \int_0^1 \int_{\bb{R}^d} \del{\vd{\sigma}{m}(\tau, \mu_\lambda,x^*( t,x,\sigma,\tau)) - \vd{\sigma}{m}(\tau, \mu_\lambda,x)}\dif m(x)\dif \lambda\\
				&= -\int_0^1 \int_{\bb{R}^d} \int_\tau^t  D_m \sigma(\tau, \mu_\lambda,x^*( t,x,\sigma,s)) \cdot \partial_s x^*( t,x,\sigma,s) \dif s\dif m(x)\dif \lambda.
			\end{split}
		\end{equation}
		By the regularity we have established, recalling the flow satisfied by $x_0^*$, it follows that $\sigma$ is differentiable with respect to $t$, and after dividing by $(t - \tau) \to 0$ we get
		\begin{equation}
			\partial_t \sigma(t,m) = - \int_{\bb{R}^d}D_m \sigma(t,m,x) \cdot D_p H\del{x,D_x v\del{t,x,\sigma(t,m)}}\dif m(x),
		\end{equation}
		which is the transport equation.
	\end{proof}

\begin{corollary}\label{cor:trans_master}
Suppose that $g(x,\cdot)\in C^1(\R)$, uniformly with respect to the $x$-variable. Suppose that the assumptions of Theorem \ref{thm:transport_deterministic} are in place. The unique classical solution to \eqref{eq:transport} provides the unique classical solution to the master equation \eqref{eq:master equation}.
\end{corollary}	

\begin{proof}
Since $\sigma$ is a classical solution to the transport equation \eqref{eq:transport}, the statement of this corollary simply follows from the representation formula \eqref{def:master_function}, by the formal identities \eqref{eq:correspondence}, as long as $v$ is continuously differentiable with respect to the $\sigma$-variable. Since $v$ is supposed to be a $C^{1,1}$ classical solution to the Hamilton--Jacobi equation \eqref{eq:HJ}, and $g$ is assumed to be continuously differentiable with respect to the parameter $\sigma$, this property will also be inherited by $v$. The result follows.
\end{proof}

We conclude this section by remarking that all of these arguments can be more or less straightforwardly adapted to \emph{systems} of transport equations, i.e.~with the variable $\sigma$ taking values in $\s{X} = \bb{R}^n$ for $n > 1$.
For this it suffices to replace Assumption \ref{eq:monotonicity} with the hypothesis that $c_0 I_n - \Sigma_0(\sigma,t,m)$ be a differentiable monotone vector field for some $c_0 < 1$.
Then all of the arguments of Theorem \ref{thm:transport_deterministic} go through with standard adaptations.
The examples from Section \ref{sec:examples} can also be generalized to this case with not too much difficulty.

	\section{Weak solutions to the transport equation in the absence of monotonicity}\label{sec:entropy}
	
\subsection{A model problem via conservation laws} \label{sec:conservation laws}
	
	Suppose that the Hamiltonian $H$ depends solely on the momentum variable, i.e. $H(x,p) = H(p)$ and $g(x,\sigma) = f(\sigma)\cdot x$ for some $f:\bb{R} \to \bb{R}^d$. Let $L:\R^d\to\R$ be the Lagrangian associated to $H$, i.e. $L=H^*$.
	The Hopf--Lax formula yields
	\begin{equation} \label{eq:a v conservation}
		x_0^*( t,x,\sigma) = x-t DH(f(\sigma)), \quad v( t,x,\sigma) = t \left\{L(DH(f(\sigma))) - f(\sigma)\cdot DH(f(\sigma))\right\} +f(\sigma) \cdot x.
	\end{equation}
	The transport equation \eqref{eq:transport} becomes
	\begin{equation} \label{eq:conservation}
		\left\{
		\begin{array}{ll}
			\partial_t \sigma(t,m) + \ds\int_{\R^d}DH\del{f(\sigma(t,m))} \cdot D_m \sigma(t,m)(\xi) \dif m(\xi) = 0, & {\rm{in}}\ (0,T)\times\sP_2(\R^d),\\
			\sigma(0,m) = \sigma_0(m), & {\rm{in}}\ \sP_2(\R^d).
		\end{array}
		\right.
	\end{equation}
	Define the following ``divergence'' operator: for any $G:\sP_2(\bb{R}^d) \to \bb{R}^d$, set
	\begin{equation} \label{def:divergence}
		\operatorname{div}_m G(m)(\xi) = \sum_{j=1}^d D_m G_j(m)(\xi) \cdot e_j,
	\end{equation}
	where $(e_1,\ldots,e_d)$ are the standard basis vectors in $\bb{R}^d$.
	Let $F$ be an antiderivative of $DH\circ f$, i.e.~$F'(s) = DH(f(s))$.
	Then \eqref{eq:conservation} can be written in divergence form as
	\begin{equation} \label{eq:conservation div}
		\left\{
		\begin{array}{ll}
			\partial_t \sigma(t,m) + \ds\int_{\R^d}  \operatorname{div}_m \del{F\del{\sigma(t,m)}}(\xi) \dif m(\xi) = 0, & {\rm{in}}\ (0,T)\times\sP_2(\R^d),\\
			\sigma(0,m) = \sigma_0(m), & {\rm{in}}\ \sP_2(\R^d),
		\end{array}
		\right.
	\end{equation}

	To convince the reader that \eqref{eq:conservation div} is indeed a conservation law in infinite dimensions, let us consider the finite-dimensional projection $\sigma_N(t,x) := \sigma_N(t,x_1,\ldots,x_N) := \sigma\del{t,\frac{1}{N}\sum_{j=1}^N \delta_{x_j}}$, where $x=(x_1,\ldots,x_N)$,
	 and define $F_N :\bb{R}^{Nd} \to \bb{R}^{Nd}$ by $F_N = \underbrace{(F,\ldots,F)}_{N-\text{times}}$. %$F_N = \underset{N \ \text{times}}{(F,\ldots,F)}$.
	Then by plugging empirical measures in for $m$ in \eqref{eq:conservation}, we derive
	\begin{equation}\label{eq:conservation N}
		\partial_t \sigma_N(t,x) + \operatorname{div}\del{F_N\del{\sigma_N(t,x)}} = 0.
	\end{equation}
{Here the divergence operator in the previous equation is the classical divergence operator acting on smooth vector fields $b:\R^{Nd}\to\R^{Nd}$. In particular there is no rescaling factor in terms of $N$ involved. Indeed, let us give some more details on this fact. First, if $\sigma:(0,T)\times\sP_{2}(\R^{d})\to\R$ is smooth in the measure variable, using the previous notation, it is well-known (cf. \cite[Chapter 5]{carmona2018probabilisticII}) that by setting $m^{N}:=\frac{1}{N}\sum_{j=1}^N \delta_{x_j},$ we have
$$
D_{m}\sigma\del{t,m^{N}}(x_{i}) = N D_{x_{i}}\sigma_{N}(t,x_{1},\dots,x_{N}).
$$
Therefore, by the definition $\operatorname{div}_{m}$, we have
\begin{align*}
\int_{\R^d}  \operatorname{div}_m \del{F\del{\sigma(t,m^{N})}}(\xi) \dif m^{N}(\xi) & = \int_{\R^d} \sum_{j=1}^d D_m \del{F_{j}\del{\sigma(t,m^{N})}}(\xi)\cdot e_j \dif m^{N}(\xi) \\
& = \sum_{i=1}^{N}\sum_{j=1}^{d} F_{j}'\del{\sigma_{N}(t,x_{1},\dots,x_{N})}D_{x_{i}}\sigma_{N}(t,x_{1},\dots,x_{N})\cdot e_{j}\\
& = \operatorname{div}\del{F_N\del{\sigma_N(t,x)}}.
\end{align*}
}	
	
	Equation \eqref{eq:conservation N} is a classical scalar conservation law in $Nd$ dimensions.
	By its formal connection with the infinite dimensional equation \eqref{eq:conservation}, we are justified in calling the latter a scalar conservation law in the Wasserstein space.

	In this section, we define weak solutions to a conservation law of the form \eqref{eq:conservation}. It is a natural question, whether one would be able to define weak entropy solutions to \eqref{eq:conservation} by studying the corresponding entropy solutions to \eqref{eq:conservation N} in $\R^{dN}$, and then by proceeding with a limiting procedure as $N\to+\infty$. Such a philosophy, although in a completely different setting, has been recently applied in \cite{cecchin2022weak}. It turns out that in our setting such a study is not necessary, as we can exploit a dimension reduction property, that will allow to define weak entropy solutions to \eqref{eq:conservation} in a direct way. We leave the study of entropy solutions in more general cases, as in \eqref{eq:transport}, to future study. Instead, in this section we study a couple {of} examples and point out a very sobering fact: entropy solutions to the transport equation in general cannot be used to select specific MFG Nash equilibria. Moreover, there are examples where Nash equilibria do not even exist, despite the fact that entropy solutions do.
	
	\subsection{A major dimension reduction}\label{subsec:dim_reduction}
	
	Equation \eqref{eq:conservation} has a very nice structure, in the sense that wave velocities exist in a finite-dimensional subspace, as we will now see.
	Consider the lifted version to the Hilbert space $\bb{H}$.
	\begin{equation} \label{eq:conservation lifted}
	\left\{
	\begin{array}{ll}
		\partial_t \tilde \sigma(t,X) + DH\del{f(\tilde \sigma(t,X)} \cdot \bb{E}[D_X \tilde \sigma(t,X)] = 0, & {\rm{in}}\ (0,T)\times\bb{H},\\ 
		 \tilde\sigma(0,X) = \tilde \sigma_0(X), & {\rm{in}}\ \bb{H}.
	\end{array}
	\right.
	\end{equation}
	Notice that for any $X \in \bb{H}$, $\bb{E}[X]$ is the projection of $X$ onto the $d$-dimensional subspace $\bb{H}_1$ consisting of \emph{constant} random variables, whose orthogonal complement is the subspace $\bb{H}_0$ consisting of mean zero random variables. In this sense, $\bb{H}=\bb{H}_0\oplus\bb{H}_1.$
	
	The PDE in Equation \eqref{eq:conservation lifted} is a nonlinear transport equation where the wave velocity is confined to $\bb{H}_1$, hence characteristics should be constant in the $\bb{H}_0$ component.
	
	This suggests the following solution to \eqref{eq:conservation lifted}.
	For $X \in \bb{H}$ we will write $X = X_0 \oplus x$ where $x = \bb{E}[X]$ is the projection onto $\bb{H}_1 \cong \bb{R}^d$ and $X_0 = X - \bb{E}[X]$ is the projection onto $\bb{H}_0$.
	Then for any $\tilde u:\bb{H} \to \bb{R}$ we will write, by a slight abuse of notation, $\tilde u(X) = \tilde u(X_0,x)$.
	Observe that we have the decomposition for the Fr\'echet derivative
	\begin{equation}
		D\tilde u(X) = D_{X_0}\tilde u(X_0,x) \oplus D_x\tilde u(X_0,x),
	\end{equation}
	where $D_{X_0}\tilde u(X_0,x)$ and $D_x\tilde u(X_0,x)$ are Fr\'echet derivatives on $\bb{H}_0$ and $\bb{H}_1$, respectively.
	Since $\bb{H}_1 \cong \bb{R}^d$, $D_x\tilde u(X_0,x)$ can be taken as the usual finite dimensional derivative with respect to $x \in \bb{R}^d$.
	Moreover, we have
	\begin{equation}
		\bb{E}[D\tilde u(X)] = D_x \tilde u(X_0,x).
	\end{equation}
	Thus \eqref{eq:conservation lifted} becomes
	\begin{equation}
	\left\{
	\begin{array}{ll}
		\partial_t \tilde \sigma(t,X_0,x) + DH\del{f(\tilde \sigma(t,X_0,x))} \cdot D_x \tilde \sigma(t,X_0,x) = 0, & {\rm{in}}\ (0,T)\times\bb{H},\\
		 \tilde\sigma(0,X_0,x) = \tilde \sigma_0(X_0,x), & {\rm{in}}\ \bb{H}.
	\end{array}
	\right.
	\end{equation}
	Under classical assumptions on $DH\circ f$, for each $X_0 \in \bb{H}_0$, there is a unique entropy solution $(t,x) \mapsto \tilde \sigma(t,X_0,x)$ to this equation.
	Formally, this should be the entropy solution to the master equation.
	
	We can also write this directly on the space of probability measures.
\begin{definition}\label{def:decomposition}
	Denote by $\sP_2^0(\bb{R}^d)$ the space of all $m \in \sP_2(\bb{R}^d)$ such that $\ds\int_{\R^d} \xi \dif m(\xi) = 0$.
	For any $m \in \sP_2(\bb{R}^d)$, write $x = \ds\int_{\R^d} \xi \dif m(\xi)$ and $m_0 = (\tau_{-x})_\sharp m \in \sP_2^0(\bb{R}^d)$. Here, for $b\in\R^d$, $\tau_b:\R^d\to\R^d$ stands for the translation map by the vector $b$, i.e. $\tau_b(x)=x+b$.
\end{definition}
\begin{remark}
	Note that $m$ can be uniquely recovered from $x$ and $m_0$ by setting $m = (\tau_x)_\sharp m_0 = \delta_x \ast m_0$ (the convolution of $\delta_x$ and $m_0$).
\end{remark}	
	Now for any $u:\sP_2(\bb{R}^d) \to \bb{R}$ write $u(m) = u(m_0,x)$.
	Then Equation \eqref{eq:conservation} can be written as
	\begin{equation} \label{eq:conservation finite dim}
	\left\{
	\begin{array}{ll}
		\partial_t \sigma(t,m_0,x) + (DH\circ f)\del{\sigma(t,m_0,x)} \cdot D_x \sigma(t,m_0,x) = 0, & {\rm{in}}\ (0,T)\times\sP_2^0(\R^d)\times\R^d,\\ 
		\sigma(0,m_0,x) = \sigma_0(m_0,x), & {\rm{in}}\ \sP_2^0(\R^d)\times\R^d
	\end{array}
	\right.
	\end{equation}
	Note that the general transport equation \eqref{eq:transport} cannot be written in this way because the vector field $y\mapsto D_pH(y,D_xv(t,y,\sigma))$ is not independent of $y$.
	
\begin{definition}\label{def:entropy_sol}	
	We say that $\sigma$ is an \emph{entropy solution} to Equation \eqref{eq:conservation} provided that $(t,x) \mapsto \sigma(t,m_0,x)$ is an entropy solution to Equation \eqref{eq:conservation finite dim} for every $m_0 \in \sP_2^0(\bb{R}^d)$.
\end{definition}

\begin{remark}
We remark that a certain dimension reduction technique has been recently used in \cite{lasry2022dimension} to study both MFG systems and master equations. That approach, although similar in spirit, is completely different from the one that we consider here. In \cite{lasry2022dimension} the authors remain the the realm of classical solutions.
\end{remark}

	\subsubsection{Existence of solutions}\label{subsec:existence}

	In addition to the assumptions of Section \ref{sec:conservation laws}, we also assume the vector-valued function $DH\circ f:\bb{R} \to \bb{R}^d$ has only a one-dimensional range.
	Thus we may replace the vector-valued function $DH\circ f$ with $\bar f\zeta$, i.e. $DH(f(u))=\bar f(u)\zeta,\ \forall \ u\in\R$, where $\bar f:\bb{R} \to \bb{R}$ and $\zeta \in \bb{R}^d$ are given, and we impose (without loss of generality) that $\abs{\zeta} = 1$.  Then we may write \eqref{eq:conservation} as
	\begin{equation} \label{eq:conservation 1}
	\left\{
	\begin{array}{ll}	
		%\begin{split}
			\partial_t \sigma(t,m) + \ds\int_{\bb{R}^d} \bar f\del{\sigma(t,m)}\zeta \cdot D_m \sigma(t,m)(\xi)\dif m(\xi) = 0, & {\rm{in}}\ (0,T)\times\sP_2(\R^d),\\
			\sigma(m,0) = \sigma_0(m), & {\rm{in}}\ \sP_2(\R^d).
		%\end{split}
	\end{array}
	\right.
	\end{equation}
	%We will assume (without loss of generality) that $\abs{\omega} = 1$.
	Defining $F(u) = \int_0^u \bar f(s)\dif s$ and using definition \eqref{def:divergence}, we may also write \eqref{eq:conservation 1} in ``divergence form'' as
	\begin{equation} \label{eq:conservation 1 div}
	\left\{
	\begin{array}{ll}
		%\begin{split}
			\partial_t \sigma(t,m) + \ds\int_{\bb{R}^d} \operatorname{div}_m\del{F\del{\sigma(t,m)}\zeta}(\xi)\dif m(\xi) = 0, & {\rm{in}}\ (0,T)\times\R^d\\
			\sigma(m,0) = \sigma_0(m), & {\rm{in}}\ \R^d.
		%\end{split}
	\end{array}
	\right.
	\end{equation}

	The entropy solution to \eqref{eq:conservation 1}, which by definition is the entropy solution to \eqref{eq:conservation finite dim}, can be obtained using the Lax--Oleinik formula, as we will see below.
	As a corollary of the formula, we will see that for a.e.~$x \in \bb{R}^d$, the unique entropy solution gives a mean field Nash equilibrium for every $m$ with mean $x$. The main theorem of this subsection can be formulated as follows.
	
\begin{theorem}\label{thm:entropy_existence}
Assume that $\bar f:\R\to\R$ is {differentiable and} uniformly increasing, i.e. there exists $c_0>0$ such that $\bar f'(s) \geq c_0$ for all $s\in\R$. Assume furthermore that $\sigma_0$ is fully $\s{C}^1$ and Lipschitz continuous with respect to $W_2$. Then, the problem \eqref{eq:conservation 1} has a unique entropy solution in the sense of Definition \eqref{def:entropy_sol}.
\end{theorem}
	
Before proving this theorem we need the following technical lemma.
	
	\begin{lemma} \label{lem:antider}
		Let $\phi$ be a fully $\s{C}^1$ function on $\sP_2(\bb{R}^d)$ which is Lipschitz continuous with respect to $W_2$, and let $\zeta$ be a fixed unit vector in $\bb{R}^d$.
		Then there exists a fully $\s{C}^1$ function $\psi$ on $\sP_2(\bb{R}^d)$ such that
		\begin{equation}
			\int_{\bb{R}^d} D_m \psi(m)(x)\cdot \zeta \dif m(x) = \phi(m),\ \ \forall\ m\in\sP_2(\R^d).
		\end{equation}
	\end{lemma}

	\begin{proof}
		Let $\tilde \phi :\bb{H} \to \bb{R}$ be the lifted version of $\phi$, i.e.~$\tilde \phi(X) = \phi(X_\sharp\bb{P})$.
		We denote by $\bb{H}^0_\zeta$ the closed subspace of all $X \in \bb{H}$ such that $\bb{E} [X \cdot \zeta] = 0$.
		For any $X \in \bb{H}$ define
		\begin{equation}
			\tilde \psi(X) := \int_0^{\bb{E} [X \cdot \zeta]} \tilde \phi\del{s\zeta + X - \bb{E} [X \cdot \zeta]\zeta}\dif s,
		\end{equation}
		noting that $X - \bb{E} [X \cdot \zeta]\zeta$ is the projection of $X$ onto $\bb{H}^0_\zeta$.
		
		We first claim that {$\tilde \psi(X)$} depends only on the law of $X$.
		Suppose $X \sim Y$, i.e.~$X$ and $Y$ have the same law.
		Then $\bb{E} [X \cdot \zeta] = \bb{E} [Y \cdot \zeta]$, from which it follows that 
		\begin{equation*}
			s\zeta + X - \bb{E} [X \cdot \zeta]\zeta \sim s\zeta + Y - \bb{E} [Y \cdot \zeta]\zeta \quad \forall s.
		\end{equation*}
		Since $\tilde \phi$ depends only on the law of its argument, the claim follows.
		
		Now we show that $\tilde \psi$ is Fr\'echet differentiable.
		Start with
		\begin{equation}
			\begin{split}
				\tilde \psi(Y) - \tilde \psi(X) &= \int_{\bb{E} [X \cdot \zeta]}^{\bb{E} [Y \cdot \zeta]} \tilde \phi\del{s\zeta + Y - \bb{E} [Y \cdot \zeta]\zeta}\dif s\\
				&\quad + \int_0^{\bb{E} [X \cdot \zeta]} \del{\tilde \phi\del{s\zeta + Y - \bb{E} [Y \cdot \zeta]\zeta} - \tilde \phi\del{s\zeta + X - \bb{E} [X \cdot \zeta]\zeta}}\dif s.
			\end{split}
		\end{equation}
		Using the continuity of $\tilde \phi$ and the fact that $\abs{\bb{E} [(Y-X) \cdot \zeta]} \leq \enVert{X-Y}$, we deduce that the first integral on the right-hand side is equal to \begin{equation}
			\tilde \phi\del{X}\bb{E} [(Y-X) \cdot \zeta]
			+  o(\enVert{X-Y})
		\end{equation} as $Y \to X$.
		As for the second integral, we use the continuity of $D\tilde \phi$ and the fact that (by the property of the projection onto $\bb{H}^0_\zeta$)
		\begin{equation*}
			\enVert{X - Y - \bb{E} [(X-Y) \cdot \zeta]\zeta} \leq \enVert{X-Y},
		\end{equation*}
		to get
		\begin{equation}
			\begin{split}
				\int_0^{\bb{E} [X \cdot \zeta]} &\del{\tilde \phi\del{s\zeta + Y - \bb{E} [Y \cdot \zeta]\zeta} - \tilde \phi\del{s\zeta + X - \bb{E} [X \cdot \zeta]\zeta}}\dif s\\
				&= \int_0^{\bb{E} [X \cdot \zeta]} \bb{E}\sbr{D\tilde \phi\del{s\zeta + X - \bb{E} [X \cdot \zeta]\zeta} \cdot \del{Y - X - \bb{E} [(Y-X) \cdot \zeta]\zeta}}\dif s\\
				& \quad + o(\enVert{X-Y}).
			\end{split}
		\end{equation}
		It follows that $\tilde \psi$ is differentiable with
		\begin{multline}
			D\tilde \psi(X) = \tilde \phi(X)\zeta + \int_0^{\bb{E} [X \cdot \zeta]}
			 D\tilde \phi\del{s\zeta + X - \bb{E} [X \cdot \zeta]\zeta}\dif s
			- \int_0^{\bb{E} [X \cdot \zeta]}\bb{E}\sbr{D\tilde \phi\del{s\zeta + X - \bb{E} [X \cdot \zeta]\zeta} \cdot \zeta}\zeta \dif s,
		\end{multline}
		which can be written more simply as
		\begin{equation}
			D\tilde \psi(X) = \tilde \phi(X)\zeta + \operatorname{proj}_{\bb{H}^0_\zeta} \int_0^{\bb{E} [X \cdot \zeta]}
			D\tilde \phi\del{s\zeta + \operatorname{proj}_{\bb{H}^0_\zeta} X}\dif s.
		\end{equation}
		We note that the Lipschitz continuity of $\phi$ with respect to $W_2$ readily implies the Lipschitz continuity of $\tilde\phi$ in $\mathbb H$. Because $D\tilde \phi$ and $\tilde\phi$ are Lipschitz continuous, so is $D\tilde \psi(X)$. 
		
		Now, let us define $\psi:\sP_2(\R^d)\to\R$ by $\psi(m) := \tilde \psi(X)$ for any $X$ whose law is $m$ (which is well-defined as $\tilde\psi$ depends only on the law of $X$); then $\psi$ is fully $\s{C}^1$.
		Moreover, since
		\begin{equation}
			\bb{E}\sbr{D\tilde \psi(X) \cdot \zeta} = \tilde \phi(X),
		\end{equation}
		we have
		\begin{equation}
			\int_{\bb{R}^d} D_m \psi(m)(x) \cdot \zeta \dif m(x) = \phi(m),
		\end{equation}
		as desired.
	\end{proof}

\begin{proof}[Proof of Theorem \ref{thm:entropy_existence}]

	Let $\phi_0:\sP_2(\R^d)\to\R$ be a fully $\s{C}^1$ function such that 
	$$\int_{\R^d} D_m \phi_0(m)(x) \cdot \zeta \dif m(x) = \sigma_0(m),\ \ \forall m\in\sP_2(\R^d),$$ 
	and let $\tilde \phi_0$ be its lifted version.
	Using the decomposition $m=(m_0,x)$ from Definition \ref{def:decomposition}, define
	\begin{equation} \label{eq:hopf lax x}
		\phi(t,m_0,x) := \inf_{\alpha \in \bb{R}} \cbr{tF^*(\alpha) + \phi_0(m_0,x - t\alpha\zeta)}.
	\end{equation}
	Thus, for any $m_0\in \sP_2^0(\bb{R}^d),$ $\phi$ is the unique viscosity solution of
	\begin{equation}
	\left\{
	\begin{array}{ll}
		\partial_t\phi(t,m_0,x) + F\del{D_x \phi(m_0,x,t) \cdot \zeta} = 0, & {\rm{in}}\ (0,T)\times\R^d,\\ 
		\phi(0,m_0,x) = \phi_0(m_0,x), & {\rm{in}}\ \bb{R}^d.
	\end{array}
	\right.
	\end{equation}
	By standard arguments (cf.~\cite[Section 3.4]{evans10}), it follows that $\sigma(t,m_0,x) := D_x \phi(t,m_0,x) \cdot \zeta$ is the unique entropy solution of
		\begin{equation} \label{eq:conservation 1x}
		%\begin{split}
		\left\{	
		\begin{array}{ll}
			\partial_t \sigma(t,m_0,x) + \bar f\del{\sigma(t,m_0,x)}\zeta \cdot D_x \sigma(t,m_0,x) = 0, & {\rm{in}}\ (0,T)\times\R^d,\\
			\sigma(0,m_0,x) = D_x\phi_0(m_0,x) \cdot \zeta, & {\rm{in}}\ \R^d.
		\end{array}
		\right.
		%\end{split}
	\end{equation}
	Notice that, by taking $X \sim m$ or equivalently $X_0 \sim m_0$, we have
	\begin{equation}
		D_x \phi_0(m_0,x) \cdot \zeta = D_x \tilde \phi_0(X_0,x)\cdot\zeta
		= \bb{E}[D_X \tilde \phi_0(X) \cdot \zeta] = \int_{\bb{R}^d} D_m \phi_0(m)(\xi) \cdot \zeta \dif m(\xi)
		= \sigma_0(m) = \sigma_0(m_0,x).
	\end{equation}
	Thus Equation \eqref{eq:conservation 1x} is exactly \eqref{eq:conservation finite dim}, so $\sigma(t,m)$ is the unique entropy solution to \eqref{eq:conservation finite dim}.

\end{proof}	
	
\subsubsection{Entropy solutions provide Nash equilibria in the case of smooth initial data and convex flux functions}	
Under the assumptions of the previous subsection, we establish that $\sigma$ does indeed give an equilibrium to the mean field game. {As we could see early on in the manuscript (see \eqref{eq:equil cond single valued}), in our setting MFG Nash equilibria are precisely described by the push forward of the initial measure through the characteristics of the transport equation for $\sigma$. When we have a classical solution to the transport equation, then characteristics are unique and reach everywhere. However, in the case of weak entropy solutions, there can even be regions without any characteristics (due to rarefaction waves). This situation happens if either the initial datum $\sigma_{0}$ has discontinuities or the flux function is non-convex, allowing rarefaction waves to form at later times. In such scenarios, MFG Nash equilibria are not selected by the entropy solution to the transport equation. Here we are using the word ``selection'' to mean that the solution $\sigma(t,m)$ should be equal to a true Nash equilibrium (of which there could be several), which as we have seen corresponds to values of $\sigma_0$ propagated along characteristics. A similar philosophy was present present in the works \cite{cecchin2019concergence, delarue2020selection}, when MFG Nash equilibria selection was given by limits of optimal trajectories to the corresponding $N$-player games. However, if the initial distribution sits on a discontinuity, then the limit of the $N$-player game and the vanishing viscosity should converge to a ``randomization'' of the two extremal MFG Nash equilibria. Our approach, although similar, is different from the point of view that there is no involvement of the $N$-player game, but rather a dimension reduction could be applied for the new transport equation for $\sigma$, due to the special structure of the data $H$ and $g$ in the MFG.}
\begin{theorem}\label{thm:selection}
Suppose that we are in the setting of the previous subsection and in particular there exists $c_0>0$ such that $\bar f(s)\ge c_0$ for all $s\in\R$. Suppose moreover that $\sigma_0$ is fully $\s{C}^1$ and Lipschitz continuous with respect to $W_2$. Then $\sigma$, the unique entropy solution to \eqref{eq:conservation}, selects a Nash equilibrium for the mean field game.
\end{theorem}
\begin{proof}
	It is enough to show that for a.e.~$x \in \bb{R}^d$, 
	\begin{equation}
		\sigma(t,m) = \sigma_0\del{x_0^*(\sigma(t,m),\cdot,t)_\sharp m}, \quad \forall m \in \sP_2(\bb{R}^d) \ \text{s.t.} \ \int_{\R^d} \xi \dif m(\xi) = x.
	\end{equation}
	Since we are in the simple case in which Equation \eqref{eq:a v conservation} holds, the equilibrium condition becomes
	\begin{equation}
		\sigma(t,m) = \sigma_0\del{\left(\tau_{-t \bar f\del{\sigma(t,m)} \cdot \zeta}\right)_\sharp m}.
	\end{equation}
	Notice that a push-forward by a constant vector changes only the average $x$ by that vector and leaves $m_0$ unchanged.
	Thus, we simply have to show
	\begin{equation}\label{eq:displacement_of_mean}
		\sigma(t,m_0,x) = \sigma_0\del{m_0,x-t \bar f\del{\sigma(m_0,x,t)}\zeta}. %\quad \text{a.e.} \ x.
	\end{equation}
	Let $\alpha(m_0,x,t)$ be the minimizer in Equation \eqref{eq:hopf lax x}.
	The first-order condition for $\alpha$ implies
	\begin{equation}
		\bar f^{-1}\del{\alpha(m_0,x,t)} = D_x \phi_{0}\del{m_0,x - t\alpha(m_0,x,t)\zeta}\cdot \zeta = \sigma_{0}\del{m_0,x - t\alpha(m_0,x,t)\zeta}
	\end{equation}	
	for a.e.~$x$.
	Since $\sigma(m_0,x,t) = \bar f^{-1}\del{\alpha(m_0,x,t)}$, this is the desired identity.
	The proof is complete.
\end{proof}

\subsection{Entropy solutions via vanishing common noise}	
We suppose that we are in the setting of Subsection \ref{subsec:dim_reduction}. In this subsection we study randomization of the equilibria by adding a stochastic forcing in the direction of the mean of the probability measures. This will correspond to a common noise at the level of the mean field game. A similar philosophy was used in \cite{delarue2020selection, tchuendom2018uniqueness} to restore uniqueness in linear quadratic MFGs. In \cite{delarue2019restoring} Delarue has also studied randomization of the equilibria by adding an infinite dimensional stochastic forcing.
		Indeed, the mean field game considered in this section almost fits the setting of \cite{delarue2019restoring} by removing the measure dependence from the running cost and the dynamics, except that we consider a more general Hamiltonian.
		Delarue also shows that the equilibrium is given by characteristics of a conservation law by taking (in our terminology) $\sigma_0(m) = \int_{\R^d} \xi \dif m(\xi)$ and a quadratic Hamiltonian.
		In the present work we assume a more general structure on the Hamiltonian and, especially, on the scalar function $\sigma_0(m)$.

%{\color{blue}[@Jameson: I don't know whether should we give more explanation on the arguments below. In particular, I don't really see how to `reply' to the points 12. and 13. from the referee report. If we add a Laplacian in the $x$-variable in the transport equation, this would correspond philosophically to a stochastic forcing via a common noise, as we tried to highlight this in the two FBSDE systems below. Regarding point 13., I have tried to carefully check the last 3 recent papers of Cardaliaguet--Souganidis (published in AAP, SIMA and NoDEA), but I didn't find the well-posedness result told by the referee. In particular, in some of those papers they construct `monotone' solutions in the sense of LL monotonicity, which is not the case in our setting in general...]}

%{[@Alpar: I don't know how to `reply' to points 12 and 13, either. Point 12 doesn't make sense to me. Yes, the strategies do depend on $B$. That is reflected in the diffusion term. Point 13 seems to be false, as far as I can tell.]
%}

We consider the problem \eqref{eq:conservation finite dim}, and we consider adding a Brownian noise $(B_\tau)_{\tau\in(-T,0)}$ (with intensity $\varepsilon>0$) in the $x$-direction (which corresponds to the direction of the mean of the measure variable). At the transport equation level this reads as
\begin{equation} \label{eq:vanishing_common}
	\left\{
	\begin{array}{ll}
		\partial_t \sigma(t,m_0,x) + (DH\circ f)\del{\sigma(t,m_0,x)} \cdot D_x \sigma(t,m_0,x) - \varepsilon\Delta_x\sigma(t,m_0,x) = 0, & {\rm{in}}\ (0,T)\times\sP_2^0(\R^d)\times\R^d,\\ 
		\sigma(0,m_0,x) = \sigma_0(m_0,x), & {\rm{in}}\ \sP_2^0(\R^d)\times\R^d.
	\end{array}
	\right.
	\end{equation} 
Under suitable assumptions on $DH\circ f$, this PDE has a unique classical solution for all $\varepsilon>0$; see e.g.~\cite[Theorem V.8.1]{ladyzhenskaia1968linear}. Moreover, if we are in the setting of Theorem \ref{thm:entropy_existence} together with the assumptions that $\bar f$ and $\sigma_0$ are uniformly bounded, the results in \cite{kruvzkov1970first} imply that this solution $\sigma_\varepsilon$ converges, as $\varepsilon\to 0$, to the unique entropy solution $\sigma$, given in Theorem \ref{thm:entropy_existence}. %{This convergence is understood classically, i.e. strongly in $L^{1}(\R^{d})$. We note that this is for all $m_{0}$ fixed, and in general we might not expect the convergence to be uniform with respect to $m_{0}$. [@Jameson: probably we might expect some uniformity for this convergence, but for that we might need stronger assumptions on $\sigma_{0}$, probably? I don't know well enough the literature on this stability, though\dots]}
{ More precisely, we have that $\sigma_\varepsilon(\cdot,m_0,\cdot) \to \sigma(\cdot,m_0,\cdot)$ in $C([0,T];L^1_{loc}(\bb{R}^d))$ for each fixed $m_0$.}

\begin{comment}
Alpar:[do we need more discussion, details here?]
Jameson:I don't think so.
\end{comment}

It is well-known (cf. \cite{pardoux1999forward}) that the PDE in \eqref{eq:vanishing_common} is strongly related to the system of FBSDEs %We use the notation $h:\R\to\R^d$, $h(r):=DH(f(r))$.
\begin{equation}\label{eq:FBSDE_vanishing}
\left\{
\begin{array}{ll}
x^{t,\xi}_s &= \xi - \ds\int_{-t}^s (DH\circ f)(y^{t,\xi}_\tau)\dif\tau+\sqrt{2\varepsilon}\int_{-t}^s\dif B_\tau,\\
y^{t,\xi}_s &= \sigma_0(m_0,x^{t,\xi}_0) + \ds\int_{s}^0 z^{t,\xi}_\tau\cdot\dif B_\tau,
\end{array}
\right.
\end{equation}
for $t\in(0,T),$ $s\in(-t,0)$ and $\xi\in\R^d$. If $DH\circ f$ and $\sigma_0(m_0,\cdot)$ are Lipschitz continuous, then \cite{pardoux1999forward} establishes the existence and uniqueness of an adapted solution $(x^{t,\xi},y^{t,\xi},z^{t,\xi})$ of \eqref{eq:FBSDE_vanishing}. Furthermore by \cite[Theorem 5.1]{pardoux1999forward}, the correspondence between viscosity solutions $\sigma$ to \eqref{eq:vanishing_common} and the solution to \eqref{eq:FBSDE_vanishing} is given by
$$\sigma(t,m_0,\xi)=y^{t,\xi}_t.$$

Now, let us consider the mean field game and perturb the trajectories of individual agents with a Brownian noise  $(B_\tau)_{\tau\in(-T,0)}$ (having intensity $\varepsilon>0$), acting in the direction of the mean of the measure component. For any initializing measure $m\in\sP_2(\R^d)$, 
using the stochastic Pontryagin principle (see \cite{carmona2013probabilistic}), one can characterize this MFG with the FBSDE system
\begin{equation}\label{mfg:fbsde_1}
\left\{
\begin{array}{ll}
X^{t,\zeta}_s &=\zeta- \ds\int_{-t}^s DH(Y^{t,\zeta}_\tau)\dif\tau + \sqrt{2\varepsilon}\int_{-t}^s\dif B_\tau,\\
Y^{t,\zeta}_s &=f(\sigma_0(\mathcal{L}_{\mathcal{F}^B_0}(X^{t,\zeta}_0))) + \ds\int_{s}^0 Z^{t,\zeta}_\tau\dif B_\tau,
\end{array}
\right.
\end{equation}
for $t\in(0,T),$ $s\in(-t,0)$ and $\zeta\in\bb{H}$ such that $\zeta_\sharp\bb{P}=m$. Let us recall that $\bb{H}$ stands for the space of $L^2$ random variables on a rich enough atomless probability space, which supports in particular $(B_\tau)_{\tau\in[-t,0]}$. Here $(\mathcal{F}^B_\tau)_{\tau\in[-t,0]}$ is the filtration generated by the common noise $(B_\tau)_{\tau\in[-T,0]}$ and $\mathcal{L}_{\mathcal{F}^B_\tau}$ stands for the conditional expectation with respect to ${\mathcal{F}^B_\tau}$.

{Let us emphasize that the stochastic forcing appearing in the FBSDE system \eqref{eq:FBSDE_vanishing}, which is linked to the transport equation, it is not directly related to the randomization of the agent trajectories in the mean field game. We can see this from the fact that the two systems \eqref{eq:FBSDE_vanishing} and \eqref{mfg:fbsde_1} are not same same. However, there is a precise relationship between these systems as we explain below.}

The global in time well-posedness of the MFG system \eqref{mfg:fbsde_1} is in general not known in the literature. This well-posedness has been established in the presence of non-degenerate idiosyncratic noise together with either the LL monotonicity condition on the data (cf. \cite{carmona2018probabilisticII}) or in the case of displacement monotone data and quadratic Hamiltonians in \cite{ahuja2016wellposedness}. In the case of LL monotone data but in the absence of non-degenerate idiosyncratic noise, the recent work \cite{cardaliaguet2022on} constructs suitable weak solutions to MFG systems. {In our case, in general the data does not possess the LL monotonicity condition.} We also notice that the approach taken in \cite{delarue2020selection} is not applicable in our case, as \eqref{mfg:fbsde_1} does not have a linear quadratic structure which was used in that work. %{\color{blue}The MFG system described in \eqref{mfg:fbsde_1} is an example of MFG driven purely by a Brownian common noise. Problems of this type have been recently  }

However, because of the the structure of the problem, we can infer the well-posedness of \eqref{mfg:fbsde_1} from the \eqref{eq:FBSDE_vanishing}, {at least under further assumptions on the data}. Indeed, let us notice that since $Y^{t,\zeta}$ needs to be adapted to the filtration generated by $B$, the random variable $X^{t,\zeta}_0$ will be just a shift of $\zeta$ by a random vector. Therefore, the law $\mathcal{L}_{\mathcal{F}^B_0}(X^{t,\zeta}_0)$ will have a random component only in the direction of the mean of $m$. Thus, it is straightforward to see that if $f$ is { an affine function}, then \eqref{eq:FBSDE_vanishing} is well-posed if and only if \eqref{mfg:fbsde_1} is well-posed, in which case we have the correspondence $Y^{t,\zeta}_s=f(y^{t,\xi}_s),$ where the relationship between $\zeta$ and $\xi$ is given by $\zeta_\sharp \mathbb{P}=m=(m_0,\xi)$.

\subsection{The entropy solution is not always an equilibrium} In this subsection we suppose that $d=1$, and suppose that the notations from Subsection \ref{subsec:existence} are in place. As we have seen in Theorem \ref{thm:selection}, in the case of strictly convex flux function and regular initial data $\sigma_0$, the unique entropy solution $\sigma$ to \eqref{eq:conservation} always selects a Nash equilibrium for the mean field game. Now, our purpose in this subsection is to show that in the cases when the strict convexity of the flux function or the continuity of the initial data are violated, the entropy solution in general does not give a Nash equilibrium. %More strikingly, below we present some examples where MFG Nash equilibria do not even exist, despite the fact that entropy solutions to the transport equation do.
\begin{comment}
Jameson:I have deleted this last sentence because I personally don't find this more striking than the other cases. In fact, I think it is far more interesting when there is at least one equilibrium, but the entropy solution does not select any of them. This seems to be the "warning" case, telling us that entropy solutions have a real limitation in their application to the game.
\end{comment}
	
\subsubsection{Discontinuous initial data, nonexistence of MFG Nash equilibria}

Let 
$$(DH\circ f)(r)=\bar f (r) = r,$$ 
{(which would correspond, for instance, to $H(p) = p^{2}/2$ and $f(r)=r$)}
so that for any $m_0$, \eqref{eq:conservation finite dim} corresponds to Burgers' equation. Consider the mean field game on the time interval $(0,T)$ with $m\in\sP_2(\R)$ as initializing measure. We represent $m$ via its decomposition $(m_0,x)$, where $x = \int_\R \xi \dif m(\xi)$. For this $m_0$ fixed, consider the Riemann problem associated to \eqref{eq:conservation finite dim} with initial datum 
$$\sigma_0(m_0,y)=
\left\{
\begin{array}{ll}
0, & {\rm{if}}\ y\le 0,\\
1, & {\rm{if}}\ y\ge 0.
\end{array}
\right.
$$
The unique entropy solution to \eqref{eq:conservation finite dim} is immediately given by 	
$$\sigma(t,m_0,y)=
\left\{
\begin{array}{ll}
0, & {\rm{if}}\ y\le 0,\\
y/t, & {\rm{if}}\ y\ge t,\\
1, & {\rm{if}}\ y< t.
\end{array}
\right.
$$	
We notice in particular that rarefaction waves have to be introduced in the time-space wedge 
$$W:=\{(t,y):\ 0<t<y\}.$$	
Let us recall that the Nash equilibrium to the mean field game is fully characterized by the relation \eqref{eq:displacement_of_mean}. Now consider $m=(m_0,x)$ such that $(T,x)\in W$. Clearly, this point is not on any of the characteristic lines (being in the region governed by rarefaction waves), and therefore for such initializing measures in the mean field game there will not be any solutions. Hence there are no Nash equilibria, yet the entropy solution to \eqref{eq:conservation finite dim} is well-defined.

\subsubsection{Discontinuous initial data, unique Nash equilibrium does not correspond to the entropy solution}

Now let us suppose that
$$(DH\circ f)(r)=\bar f (r) = r^2,$$
{(corresponding, for instance, to $H(p)=p^{2}/2$ and $f(r)=r^{2}$)}
and so \eqref{eq:conservation finite dim} corresponds to a conservation law with the non-convex flux function $\frac13 r^3$.  Using the decomposition $m=(m_0,x)$ for the initializing measure $m$, at time $T$, for $m_0$ fixed consider 
$$\sigma_0(m_0,y)=
\left\{
\begin{array}{ll}
-1, & {\rm{if}}\ y\le 0,\\
1, & {\rm{if}}\ y\ge 0.
\end{array}
\right.
$$
Since the flux function is not uniformly convex, it is well-known that Lax's entropy condition is not enough to determine the unique entropy solution. The right condition imposed on shocks to select the unique entropy solution is the so-called Oleinik entropy condition (cf. \cite[Chapter 8]{dafermos16})

For $m_0$ fixed, the unique entropy solution to \eqref{eq:conservation finite dim} is given by
\begin{equation}\label{eq:ent_sol_ex2}
\sigma(t,m_0,y)=
\left\{
\begin{array}{ll}
-1, & {\rm{if}}\ y\le t/4,\\
\sqrt{y/t}, & {\rm{if}}\ t/4\le y\le  t,\\
1, & {\rm{if}}\ y\ge t.
\end{array}
\right.
\end{equation}
As the wave speed in this case is given by the function $r\mapsto r^2$, and for our initial condition, $(-1)^2=1^2$, the information would travel in time along the parallel characteristics $y=s+t$, where $s\in\R$.

From the point of view of the MFG, we can compute the unique Nash equilibrium as follows: for any initializing measure $m\in\sP_2(\R)$ at time $T$, consider the decomposition $m=(m_0,x)$. To find the unique Nash equilibrium, it is enough to trace back the characteristic line passing through $(T,x)$, and evaluate the initial data at the base point. So this is fully characterized by $\sigma_0(m_0,x-T)$.
	
However, clearly if $m=(m_0,x)$ is such that $(T,x)$ belongs to the time-space wedge
$$W:=\{(t,y):\ t/4\le y\le  t\},$$
the entropy solution given in \eqref{eq:ent_sol_ex2} will not select the unique MFG Nash equilibrium.

\subsubsection{Smooth initial data, MFG Nash equilibrium does not correspond to the entropy solution}
The fact that entropy solutions to \eqref{eq:conservation finite dim} do not select MFG Nash equilibria is not necessarily the result of the discontinuity in the initial data $\sigma_0$. 

A general phenomenon that we can see from the definition of mean field game Nash equilibrium, i.e. from \eqref{eq:equil cond single valued}, is that when linked to the transport equation, equilibria correspond to characteristics. Therefore, if the entropy solution is such that it involves rarefaction waves, those regions can never correspond to Nash equilibria. As long as the flux function $F(r):=\int_0^r\bar f(s)\dif{s} = \int_0^r (DH\circ f)(s)\dif{s}$ is not uniformly convex, we expect this phenomenon may arise in general for continuous (or even smooth) initial data. Such results might be well-known for experts in conservation laws, but we could not locate a precise construction in the literature, for continuous initial datum. Therefore, we construct by hand such an example. As we will see below, in our construction it will be crucial that $F$ has two inflection points.

We set $F(r):=\frac{1}{12}r^4 - \frac{1}{2}r^2$ (see Figure \ref{fig:flux}). {This $F$ would for instance correspond to $H(p)=p^{2}/2$ and $f(r)=r^{3}-r$.}

Let us spend some time describing a potential initial data $\sigma_0(m_0,\cdot)$ which would produce the desired property.
(We will construct $\sigma_0(m_0,\cdot)$ having no dependence on $m_0$.)
Notice that $F$ has precisely two inflection points $r = \pm 1$.
On the interval $[-1,1]$, $F'(r) = \frac{1}{3}r^3 - r$ strictly decreases from $2/3$ to $-2/3$.
For $x \in [-2/3,2/3]$, therefore, we can select $\sigma_0(m_0,x)$ to be the unique $r \in [-1,1]$ such that $F'(r) = -x$.
Hence the characteristics originating from the interval $[-2/3,2/3]$ all meet at the point $(x=0,t=1)$.
In particular, note that $\sigma_0(m_0,\pm 2/3) = \pm 1$.
A shock $s_1(t)$ will form having initial speed of
\begin{equation*}
	\dot{s}_1(1) = \frac{F(\sigma_+) - F(\sigma_-)}{\sigma_+ - \sigma_-} = 0
\end{equation*}
because $\sigma_+ = 1$ and $\sigma_- = -1$ and $F$ is an even function.

Next, for $x \in [2/3,1]$ we choose $\sigma_0(m_0,x)$ so that it increases from 1 to $\sqrt{3}$, i.e.~so that the wave speed $F'\del{\sigma_0(m_0,x)}$ increases from $-2/3$ to 0.
Indeed, $F'$ is strictly increasing on the interval $[1,\sqrt{3}]$ into $[-2/3,0]$, so we can select $\sigma_0(m_0,x)$ to be the unique $r \in [1,\sqrt{3}]$ such that $F'(r) = 2x - 2 \in [-2/3,0]$.
In particular, the characteristics originating from a point $x_0 \in [2/3,1]$ fan out within the region between the axis $x = 1$ and the characteristic $t \mapsto \frac{2}{3}(1-t)$, having the form
\begin{equation*}
	x(t) = x_0 + t(2x_0 - 2), \quad x_0 \in [2/3,1]
\end{equation*}
and reaching the axis $x = 0$ at time $t_0 = \frac{x_0}{2-2x_0}$.
See Figure \ref{fig:characteristics}.

Now we consider $x \in [-1,-2/3]$.
We essentially want to mirror the characteristics that emanate from the interval $[2/3,1]$, so that the shock originating at $(x=0,t=1)$ will continue at a speed of zero.
However, it will in fact simplify the solution for later times if we only do this for $x \in [x^*,-2/3]$ for some $x^* \in (-1,-2/3)$ to be specified.
For now we only specify that $x^*$ will be close to $-1$.
Since $F'$ is strictly increasing on the interval $[-\sqrt{3},-1]$ into $[0,2/3]$, we select $\sigma_0(m_0,x)$ to be the unique $r \in [-\sqrt{3},-1]$ such that $F'(r) = 2x + 2 \in [0,2/3]$.
Note that $\sigma_0(m_0,x^*) \approx -\sqrt{3}$ and thus $F'(\sigma_0(m_0,x^*)) \approx 0$.
Mirroring the characteristics originating in $[2/3,1]$, the characteristics originating in $[x^*,-2/3]$ fan out between the axis $x = -1$ and the characteristic $t \mapsto \frac{2}{3}(t-1)$, having the form
\begin{equation*}
	x(t) = x_0 + t(2x_0 + 2), \quad x_0 \in [x^*,-2/3]
\end{equation*}
and reaching the axis $x=0$ at time $t_0 = -\frac{x_0}{2+2x_0}$.
See Figure \ref{fig:characteristics}.

For $x \leq x^*$, we will set $\sigma_0(m_0,x) = \sigma_0(m_0,x^*)$, so the slope of characteristics is constant (and close to zero) on the left-hand side of $x^*$.

Next consider $x \geq 1$.
We now reverse course and decrease $\sigma_0(m_0,x)$ from $\sqrt{3}$ back down to 1, so that the wave speed decreases from 0 to $-2/3$.
This can be done at any rate we like.
In particular, take some arbitrary $\xi > 1$, and then for $x \in [1,\xi]$ we take $\sigma_0(m_0,x)$ to be the unique $r \in [1,\sqrt{3}]$ such that $F'(r) = -\frac{2}{3}\frac{x-1}{\xi-1}$.
Then the characteristics emanating from this interval all intersect at the point $(x=1,t = t_\xi)$ where $t_\xi = \frac{3}{2}(\xi-1)$.
A shock $s_2(t)$ will form at this point with initial speed
\begin{equation*}
	\dot{s}_2(t_\xi) = \frac{F(\sigma_+)-F(\sigma_-)}{\sigma_+ - \sigma_-} = \frac{F(1) - F(\sqrt{3})}{1-\sqrt{3}} = -\frac{1}{3(\sqrt{3}-1)}.
\end{equation*}
We set $\sigma_0(m_0,x) = 1$ for all $x \geq \xi$, so that $\sigma_+ = 1$ at every point along the shock $s_2(t)$.
To the left of this shock we always have $1 \leq \sigma_- \leq \sqrt{3}$ because of how $\sigma_0(m_0,x)$ was chosen for $x \in [2/3,1]$.
It follows that
\begin{equation*}
	-\frac{2}{3} \leq \dot{s}_2(t) \leq -\frac{1}{3(\sqrt{3}-1)}
\end{equation*}
for all $t > t_\xi$ until shock $s_2(t)$ inevitably collides with $s_1(t) = 0$ at some time $t^* > t_\xi$.
Notice that the time $t^*$ does not depend on the choice of $\sigma_0(m_0,x)$ for any $x$ such that the characteristics $t \mapsto x + t(2x+2)$ do not intersect the shock $s_1(t)$ before $t^*$.
Hence we may choose $x^* \in (-1,-2/3)$ such that the characteristic $t \mapsto x^* + t(2x^* + 2)$ emanating from $x^*$ intersects both of these shocks at precisely the point $(0,t^*)$; namely, we set $x^* = -\frac{2t^*}{1+2t^*}$.
By picking $\xi$ large enough and therefore $t^* > t_\xi$ large enough, we can make $x^*$ arbitrarily close to $-1$.

Finally, for times $t > t^*$ we have a Riemann problem with initial function $\sigma(t^*,m_0,x)$ given by
\begin{equation*}
	\sigma(t^*,m_0,x) = \begin{cases}
		\sigma^* &\text{if}~x < 0,\\
		1 &\text{if}~x > 0
	\end{cases}
\end{equation*}
where $\sigma^* = \sigma_0(m_0,x^*)$.
If $\sigma^*$ is close enough to $-\sqrt{3}$, then we know a simple shock cannot occur, since the entropy condition is not satisfied (the line connecting $(\sigma_*,F(\sigma_*))$ to $(1,F(1))$ would not lie entirely underneath the graph of $F$).
Thus, a rarefaction wave forms.
Let $r_1$ be the least point in the interval $(-\sqrt{3},-1)$ such that $r = r_1$ solves the equation
\begin{equation*}
	\frac{F(1)-F(r)}{1-r} = F'(r).
\end{equation*}
To see that this equation has a solution in $(-\sqrt{3},-1)$, one has only to note that $F'(r) - \frac{F(1)-F(r)}{1-r}$ is negative if $r = -\sqrt{3}$ and positive if $r = -1$.
We may assume that $\sigma^* < r_1$, and therefore a rarefaction wave forms.
In particular, we have, for all $t > t^*$,
\begin{equation}
	\sigma(t,m_0,x) = \begin{cases}
		\sigma^* &\text{if}~x < F'(\sigma^*)(t-t^*),\\
		1 &\text{if}~x > F'(r_1)(t-t^*),
	\end{cases}
\end{equation}
and, for $F'(\sigma^*)(t-t^*) \leq x \leq F'(r_1)(t-t^*)$, $\sigma(t,m_0,x)$ is the unique $r \in [\sigma^*,r_1]$ such that $F'(r) = \frac{x}{t-t^*}$.
The shock $s_3(t) = F'(r_1)(t-t^*)$ satisfies the entropy condition, since $\sigma_- = r_1$ and $\sigma_+ = 1$, whereas $r_1$ was chosen precisely so that
\begin{equation*}
	\frac{F(1)-F(r_1)}{1-r_1} = F'(r_1) = \dot{s}_3(t).
\end{equation*}
All of this is explained by Figure \ref{fig:characteristics} (a picture is indeed worth 1000 words).

\begin{figure}[H]
\centering
\begin{minipage}{.5\textwidth}
  \centering
  \begin{tikzpicture}
  	\begin{axis}[axis lines=middle, samples=200]
  		\addplot[domain=-3:3]{x^4/12 - x^2/2};
  	\end{axis}
  \end{tikzpicture}	
  \captionof{figure}{The flux function $F(r)=\frac{1}{12} r^4-\frac{1}{2}r^2$}\label{fig:flux}
\end{minipage}%
\begin{minipage}{.5\textwidth}
  \centering
  \begin{tikzpicture}
  	\draw[->,gray] (-4,0) -- (4,0); %horizontal x-axis
  	\draw[->,gray] (0,0) -- (0,8); %vertical t-axis
  	%characteristics emanating from [-2/3,2/3]
  	\foreach \i in {-.66,-.55,...,.66}
  	\draw (\i,0) -- (0,1);
  	%characteristics emanating from [2/3,1] until shock
  	\foreach \j in {.66,.68,...,.92}
  	{\pgfmathsetmacro{\k}{\j/(2-2*\j)}
  	\draw (\j,0) -- (0,\k);}
	\foreach \j in {1,.98,.96,.94}
	{\pgfmathsetmacro{\t}{3 + 625*(1-\j)^2}
	\pgfmathsetmacro{\k}{\j + \t*(2*\j - 2)}
	\draw (\j,0) -- (\k,\t);}
  	%characteristics emanating from [x^*,-2/3] until shock
  	\foreach \j in {-.92,-.90,...,-.66}
  	{\pgfmathsetmacro{\k}{-\j/(2+2*\j)}
  		\draw (\j,0) -- (0,\k);}
  	%characteristics emanating from [1,\infty) until shock
  	\foreach \i in {1.2,1.4,...,4}
  	\draw (\i,0) -- (1,3);
  	\foreach \i in {2.1,2.4,...,5} 
  	{\pgfmathsetmacro{\j}{(\i-1.8)/3}
  		\pgfmathsetmacro{\k}{4+\i - \j}
  	\draw (4,\i) -- (\j,\k);
  	} 
  	\foreach \i in {5.1,5.4,...,7.6} 
  	{\pgfmathsetmacro{\k}{\i-3.8}
  		\draw (4,\i) -- (\k,7.8);
  	} 
  	\foreach \i in {0,.3,...,2.1} 
  	{\pgfmathsetmacro{\k}{\i^2/4}
  		\draw (4,\i) -- (1-\k,\i+3+\k);
  	}
  	%characteristics emanating from (-\infty,x^*] until shock
  	\foreach \j in {-3.9,-3.7,...,-1}
  	{\pgfmathsetmacro{\k}{\j + 5.8*.2}
  		\draw (\j,0) -- (\k,7.8);}
  	%shock that forms on the right
  	\draw[red,ultra thick] (1,3) .. controls (.5,5) .. (0,5.8); 
  	%shock that forms in the middle
  	\draw[red,ultra thick] (0,1) -- (0,5.8);
  	%shock that forms from the collision of two shocks
  	\draw[red,ultra thick] (0,5.8) -- (1,7.8);
  	%rarefaction wave
  	\draw[orange,ultra thick] (0,5.8) -- (.3,7.8);
  	\foreach \i in {.4,.5,...,1}
  	\draw[orange] (0,5.8) -- (\i,7.8);
  \end{tikzpicture}	
  \captionof{figure}{Black: characteristics; red: shock; orange: rarefactions}\label{fig:characteristics}
\end{minipage}
\end{figure}

When considering an initializing measure $m=(m_0,x)$ at time $T$, such that $(T,x)$ belongs to the orange region, multiple Nash equilibria for the game could be computed by letting characteristics into the orange region and tracing them back to the initial data. However, \emph{none of these equilibria} correspond to the unique entropy solution to \eqref{eq:conservation finite dim}, as this is constructed via the orange rarefaction waves. 
Here we have an explicit construction of continuous initial data for which the entropy condition therefore selects none of the possible equilibria for a certain collection of initializing measures.
This phenomenon is already well-known in the literature on conservation laws.
We refer to \cite{bianchini2017on} for a fine description of the structure of entropy solutions.
When applied to a mean field game, this implies that the entropy solution may not always be a perfect candidate for a ``selection principle'' for games with multiple equilibria.

\begin{comment}

	Let $F$ be any smooth function such that $F(u) = 0$ for $u \leq 0$, $F(u) = 1$ for $u \geq 1$, and $F$ increases smoothly on $[0,1]$.
	Set $f = F'$.
	Now let $u_0(x) = 1$ for $x > 0$ and $u_0(x) = 0$ for $x < 0$.
	Consider the unique entropy solution of
	\begin{equation}
		u_t + F(u)_x = 0, \quad u(x,0) = u_0(x).
	\end{equation}
	In this case a shock forms due to characteristics in the half-plane $x > 0$ intersecting with a rarefaction wave emanating from the discontinuity at the origin.
	We know by the Rankine-Hugoniot formula that the shock $x = \gamma(t)$ must travel at a positive speed, hence there is an open region $0 < x < \gamma(t)$ in which the rarefaction occurs and we have $u(x,t) = g(x/t)$.
	Then $g$ must satisfy $f(g(s)) = s$ for all $s$ in the range of $f$.
	Since $g$ must be continuous and $g(0) = u(0,t) = 0$, it follows that $g(s) \leq u^* < 1$ where $u^*$ is the first location in $(0,1)$ of a local maximum of $f$.
	
	Consider a smooth approximation $u_{0,1}(x)$ of $u_0$ such that $u_{0,1}(x) = u_0(x)$ for $x \notin (-1,0)$ and $u_{0,1}$ increases smoothly on $(-1,0)$.
	There is a point $x^* \in (-1,0)$ at which $f\del{u_{0,1}(x)}$ reaches its maximum.
	Let $A$ be the region enclosed by the characteristic emanating from the point $(x^*,0)$, the shock $s(t)$ described above, and the line $t = 0$.
	
	Now let $0 < \varepsilon < 1$ and $u_{0,\varepsilon}(x) = u_{0,1}(x/\varepsilon)$.
	
\end{comment}

\medskip
\medskip

\noindent {\bf Acknowledgements.}	Both authors are grateful to Alberto Bressan and Elio Marconi for the discussions about the structure of entropy solutions for scalar conservation laws in the case of non-convex flux functions. PJG acknowledges the support of the National Science Foundation through NSF Grants DMS-2045027 and DMS-1905449. We acknowledge the support of the Heilbronn Institute for Mathematical Research and the UKRI/EPSRC Additional Funding Programme for Mathematical Sciences through a focused research grant ``The master equation in mean field games''.  ARM has also been partially supported by the EPSRC via the NIA with grant number EP/X020320/1	 and by the King Abdullah University of Science and Technology Research Funding (KRF) under Award No. ORA-2021-CRG10-4674.2.
	
\medskip

\noindent {\bf Conflict of interest.} The authors declare that they do not have any conflicts of interests.	

\medskip

\noindent {\bf Data Availability Statement.} Data sharing not applicable to this article as no datasets were generated or analyzed during the current study.
	
	\bibliographystyle{plain}	
	%\bibliography{../../mybib/mybib}
	\bibliography{mybib}
\end{document}